\documentclass[11pt,a4paper]{article}
\usepackage[body={155mm,230mm},top=34mm]{geometry}
\usepackage{bm}
\usepackage{epsfig}
\usepackage{amsmath}
\usepackage{amssymb}
\usepackage{amsthm}
\usepackage{mathrsfs}
\usepackage{epsfig}
\usepackage{verbatim}
\usepackage{url}
\usepackage{hyperref,enumitem}
\usepackage{authblk}

\usepackage{accents}

\def\S{S}

\def\tr{\text{tr}}

\newtheorem{Rem}{Remark}

\newtheorem{theorem}{Theorem}

\newtheorem{lemma}{Lemma}
\newtheorem{corollary}{Corollary}
\newtheorem{definition}{Definition}
\newtheorem{remark}{Remark}

\newtheorem{proposition}{Proposition}

\def\tr{\text{tr}}

\newcommand{\red}[1]{{\color{red}#1}}
\newcommand{\blue}[1]{{\color{blue}#1}}
\newcommand{\dist}{\mbox{{\rm dist}}}
\newcommand{\cad}{c\`adl\`ag }

\def \eps{\varepsilon}
\def \bN{\mathbb{N}}
\def \bR{\mathbb{R}}

\def \bL{\mathbb{L}}
\def \bE{\mathbb{E}}
\def \bF{\mathcal{F}}
\def \bF{\mathbb{F}}
\def \cY{\mathcal{Y}}

\def \cL{\mathcal{L}}
\def \cE{\mathcal{E}}

\def \cZ{\mathcal{Z}}

\def \cF{\mathcal{F}}

\def \cY{\mathcal{Y}}

\def \tpi{\widetilde{\pi}}

\def \bP{\mathbb{P}}
\def \1{\mathbf{1}}

\newcommand{\Ymin}{Y^{\min}}
\newcommand{\Zmin}{Z^{\min}}
\newcommand{\Pmin}{\psi^{\min}}
\newcommand{\Mmin}{M^{\min}}
\newcommand{\Umin}{U^{\min}}

\newcommand{\Yminx}{Y^{\min,x}}
\newcommand{\Zminx}{Z^{\min,x}}

\title{ Continuity problem for singular BSDE with random terminal time\footnote{This work is supported by TUBITAK 
(The Scientific and Technological Research Council of Turkey) through project number 118F163. We are grateful for this support.}}

\author[1]{Sharoy Augustine Samuel \thanks{sharoys@gmail.com}}
\author[2]{Alexandre Popier\thanks{alexandre.popier@univ-lemans.fr}}
\author[3]{Ali Devin Sezer\thanks{devin@metu.edu.tr}}
\affil[1]{\small Institute of Applied Mathematics, Middle East Technical University}
\affil[2]{\small Laboratoire Manceau de Math\'ematiques, Le Mans Universit\'e, Avenue O.~Messiaen, 72085~Le~Mans cedex 9, France}
\affil[3]{\small Institute of Applied Mathematics, Middle East Technical University}

\date{\today}

\begin{document}
\maketitle

\begin{abstract}
We study a class of nonlinear BSDEs with a superlinear driver process $f$ adapted to a filtration ${\mathbb F}$ and over a random time interval $[\![0,S]\!]$ where $S$ is a stopping time of ${\mathbb F}$. The terminal condition $\xi$ is allowed to take the value $+\infty$, i.e., singular. Our goal is to show existence of solutions to the BSDE in this setting. We will do so by proving that the minimal supersolution to the BSDE is a solution, i.e., attains the terminal values with probability $1$. We consider three types of terminal values: 1) Markovian: i.e., $\xi$ is of the form $\xi = g(\Xi_S)$ where $\Xi$ is a continuous Markovian diffusion process and $S$ is a hitting time of $\Xi$
and $g$ is a deterministic function 2) terminal conditions of the form $\xi = \infty \cdot {\bm 1}_{\{\tau \le S\}}$ and 3) $\xi_2 = \infty \cdot {\bm 1}_{\{ \tau >S \}}$ where $\tau$ is another stopping time. For general $\xi$ we prove the minimal supersolution is continuous at time $S$ provided that ${\mathbb F}$ is left continuous at time $S$. We call a stopping time $S$ solvable with respect to a given BSDE and filtration if the BSDE has a minimal supersolution with terminal value $\infty$ at terminal time $S$. The concept of solvability plays a key role in many of the arguments. Finally, we discuss implications of our results on the Markovian terminal conditions to solution of nonlinear elliptic PDE with singular boundary conditions.
\end{abstract}

\vspace{0.5cm}
\noindent \textbf{2020 Mathematics Subject Classification.} 35J75, 60G40, 60G99, 60H30, 60H99.

\smallskip
\noindent \textbf{Keywords.} Backward stochastic differential equation, stopping time, singularity, continuity problem, Green's function.

\vspace{0.3cm}

\section{Introduction and definitions}

A backward stochastic differential equation (BSDE)
is a stochastic differential equation (SDE) with a prescribed terminal condition.
They have been intensively studied since the seminal papers \cite{bism:73,pard:peng:90}; they arise naturally in 
stochastic optimal control problems (see among others \cite{yong:zhou:99}), they provide a probabilistic representation of semi-linear partial differential equations (PDE) extending the Feynman-Kac formula (\cite{pard:rasc:14}) and they have found numerous applications in finance and insurance \cite{delo:13,elka:peng:quen:97}. 

If the driver term of the BSDE has superlinear growth the solution of the BSDE can blow up in finite time, this allows one to specify 
$\infty$ as a possible terminal value for such BSDE; when the terminal value is allowed to take $\infty$ it is called ``singular.''
In \cite{ahmadi2019backward,krus:popi:15,popi:06,krus:popi:seze:18}, we study nonlinear BSDE
with singular terminal condition at a deterministic terminal time $T$. 
Such BSDE generalize parabolic diffusion-reaction PDE with singular final trace (\cite{grae:hors:sere:18,marc:vero:99,popi:17}) 
and they are a key tool in optimal stochastic control problems with terminal constraints 
(\cite{anki:jean:krus:13,grae:hors:sere:18,krus:popi:15} and the references therein).

In this paper, we focus on BSDE with singular terminal conditions over a random time horizon.
We adopt the general framework for BSDE with terminal singular values established in \cite{krus:popi:14,krus:popi:15,krus:popi:17} 
and consider BSDE 
of the following form
\begin{equation} \label{eq:bsde}
dY_t  = - f(t,Y_t,Z_t,U_t) dt + Z_t dW_t + \int_\cE U_t(e) \tpi(de,dt) + d M_t, Y_S = \xi,
\end{equation}
where $W$ is a $d$-dimensional Brownian motion and $\tpi$ is a compensated Poisson random measure on a probability space $(\Omega,\cF,\bP)$ 
with a filtration $\bF = (\cF_t)_{t\geq 0}$; the unknown that is sought is the quadruple $(Y,Z,U,M).$
The filtration $\bF$ is supposed to be complete and right continuous. The solution component $M$ is required to be a local martingale orthogonal 
to $\tpi$. The function $f:\Omega\times \bR \times \bR^d \times \mathcal B^2_\mu \to \bR$ is called the {\it generator} (or {\it driver}) of 
the BSDE. 
Finally $\S$ is a stopping time of the filtration $\bF$ and $\xi$ is an ${\mathcal F}_S$ measurable random variable, which is singular,
i.e., 
${\mathbb P}(\{\xi=\infty\}) > 0.$
Precise conditions on all of these terms are spelled out in subsections
\ref{ss:integrable} and \ref{ssect:sing_supersol} below.
A quadruple $(Y,Z,U,M)$ is said to be a supersolution of \eqref{eq:bsde} if it satisfies the first equation in \eqref{eq:bsde} and
\begin{equation} \label{eq:term_cond_super_sol}
\liminf_{t\to +\infty} Y_{t\wedge \S} \geq \xi , \text{ almost surely},
\end{equation}
holds.
A supersolution $(\Ymin,\Zmin,\Umin,\Mmin)$ is called minimal if $\Ymin \le Y$ for any other supersolution $(Y,Z,U,M).$
We say $(Y,Z,U,M)$ solves the BSDE with singular terminal condition $\xi$ if it satisfies the first equation in \eqref{eq:bsde} and
\begin{equation} \label{eq:termcondsol}
\lim_{t\to +\infty} Y_{t\wedge \S}= \xi;
\end{equation}
i.e., to go from a supersolution to a solution we need to replace the $\liminf$ in \eqref{eq:term_cond_super_sol} with $\lim$ and
$\ge$ with $=$.
In the rest of this paper whenever we refer to the ``solution'' of a BSDE
with a singular terminal value, it will be in the sense of 
\eqref{eq:termcondsol}.
The condition \eqref{eq:termcondsol} means that the process $Y$ is continuous at time $\S$; for this reason we refer to the problem
of establishing that a candidate solution satisfies \eqref{eq:termcondsol} as the {\it ``continuity problem''.} 
Just as BSDE over deterministic time intervals generalize parabolic PDE,
BSDE over random time intervals are generalizations of elliptic PDE; we provide further comments on this connection,
on the motivation for the study
of BSDE over random time horizon with singular terminal values and on the implication of continuity results for BSDE theory
as well as constrained stochastic optimal control at the end of this subsection.

We call a terminal condition ``Markovian'' if it is
of the form $\xi = g(\Xi_S)$ where, 
$g:{\mathbb R}^d \mapsto {\mathbb R}_+ \cup \{\infty\}$,
 $\Xi$ is a Markov diffusion process and $S$ is the first time $\Xi$ hits 
a smooth $\partial D$, $D \subset {\mathbb R}^d$.
For such exit times, existence of minimal supersolutions for \eqref{eq:bsde}
are proved in \cite{krus:popi:15} for arbitrary terminal condition
(see subsection \ref{ssect:sing_supersol} below).
The work
\cite{popi:07} proves that these minimal supersolutions are in fact solutions
for the case ${\mathbb F}={\mathbb F}^W$
and for the specific generator $f(y) = -y|y|^{q-1}$ and for 
Markovian terminal conditions.
The works \cite{krus:popi:14,krus:popi:17} develop solutions to \eqref{eq:bsde} when $\xi$ belongs to some integrability space. 
The goal of the present work is to prove that the minimal supersolution
of \eqref{eq:bsde} satisfies \eqref{eq:termcondsol} (and therefore is a solution)
for several classes of singular terminal conditions and several assumptions on $\S$. We outline these classes and assumptions in the following paragraphs.

In two previous works \cite{krus:popi:seze:18} and \cite{ahmadi2019backward} that prove continuity results for deterministic
terminal times, two of the main ingredients are
the minimal supersolution $Y^{\min,\infty}$ 
with terminal condition $\infty$ at terminal time and the apriori upperbounds on supersolutions; both of these, are readily
available  in the prior literature for deterministic terminal times 
(for the one dimensional Brownian case treated in \cite{krus:popi:seze:18}, $Y^{\min,\infty}$ is
deterministic and has an explicit formula). For random terminal times the existence of $Y^{\min,\infty}$ and apriori upperbounds are known
only for exit times of Markov diffusions from smooth domains. 
One of the main ideas of the present work is to 
impose the existence of $Y^{\min,\infty}$ as an assumption on the stopping time $S$ and base most of our arguments on this assumption. 
We call the terminal stopping time $S$ {\em solvable} with respect
to the BSDE \eqref{eq:bsde} if there exists a supersolution to the BSDE with terminal value $\infty$ at terminal
time $S$ (see Definition \ref{d:solvable_rtt}), 
deterministic times and exit times of
Markov diffusion processes are solvable for a wide range of BSDE;
times that have a strictly 
positive density around $0$ are not solvable \cite{krus:popi:15}.
Many of our arguments are based on this solvability concept; 
some basic consequences
of solvability are given in Section \ref{s:solvablelemmas}.
In particular, if $S$ is solvable, the BSDE \eqref{eq:bsde} has a minimal
supersolution for any singular terminal condition $\xi \ge 0$ 
(Lemma \ref{lem:solvability_min_sol}).
In addition to $S$ being solvable, in many arguments we assume $\bF$ to be left continuous at $S$ for the following reason.
Because the filtration $\bF$ is assumed to be general (apriori only completeness and right-continuity is assumed) 
there is no way to control the jumps of the additional local martingale component $M$ of the solution at the terminal time.
To avoid such jumps, we suppose that $\bF$ is left-continuous at time $S$.

We now indicate the main results of the present work.
In
Section \ref{sect:existence_limit} we assume $S$ to be solvable and consider 
the problem of proving the existence of the $\lim_{t\rightarrow \infty} Y^{\min}_t$ for an arbitrary singular terminal
condition $\xi \ge 0.$
When $\S$ is deterministic, in \cite{popi:16}, the existence of a limit for $\Ymin$ is proved under some additional conditions on generator 
$f$. Here we show that these assumptions are also sufficient for a random terminal time (Section \ref{sect:existence_limit}) provided that it is solvable.

Section \ref{sect:markov_setting} focuses on Markovian terminal conditions.
To the best of our knowledge, \cite{popi:07} is the only paper that proves continuity results for a singular terminal condition at a random time $\S$; \cite{popi:07} assumes $f(y) = -y|y|^q$, $\xi$ to be Markovian and ${\mathbb F}$ Brownian. 
The results in Section \ref{sect:markov_setting} generalize the results in \cite{popi:07} to a general filtration ${\mathbb F}$ and driver $f$
keeping the terminal condition Markovian.
An important step is a bound on the expected value of an integral over the solution processes and $\dist(\Xi)$, where
$\dist(x)=  \inf_{y \in \partial D}|x-y|$ (see Lemma \ref{lem:major_Zmin_with_g}).
One of the main ingredients in the proof is the apriori upperbound on $Y$ derived
in \cite{krus:popi:15}.
When ${\mathbb F}$ is Brownian and $f$ is deterministic, the solution of the BSDE with a Markovian terminal 
condition can be used construct a viscosity solution of an associated elliptic PDE. This is discussed in subsection
\ref{ss:relatedPDE}.

Sections \ref{sect:term_cond_xi_1} and \ref{sect:term_cond_xi_2} focus on the continuity problem  for
non-Markovian terminal conditions of the form
$\xi_1 = \infty \cdot {\bm 1}_{\{\tau \le S\}}$ (Section \ref{sect:term_cond_xi_1}) and
 $\xi_2 = \infty \cdot {\bm 1}_{\{\tau > S\}}$ (Section \ref{sect:term_cond_xi_2})
 where $\tau$ is another stopping time of ${\mathbb F}.$
The results in these sections generalize results from \cite{krus:popi:seze:18} (the one dimensional
Brownian case) and \cite{ahmadi2019backward} (the general filtration, driver case) treating same type
of terminal conditions where $S$ is assumed to be deterministic. 
Events of the form $\{\tau \le S \}$ naturally arise when one modifies constraints on stochastic
optimal control problems based on the values the state process of the problem takes.
We refer to \cite{krus:popi:seze:18, ahmadi2019backward} for more comments
on why we pay particular attention to these type of non-Markovian terminal conditions. Solution of the continuity
problem for general
terminal conditions of the form $\infty \cdot {\bm 1}_A$ for arbitrary $A \in {\mathcal F}_S$ is an open problem even
for the one dimensional Brownian case and $S$ deterministic.

Section \ref{sect:term_cond_xi_1} provides two arguments to prove 
\begin{equation}\label{eq:cont}
\lim_{t\to +\infty} \Ymin_{t\wedge \S}= \xi_1.
\end{equation}
The first one is an adaptation of the argument given for the same type of terminal condition in \cite{ahmadi2019backward}.
It involves the construction of an auxiliary linear process that dominates $\Ymin$ and that is known to have
the desired limit property at the terminal time $S$.  The main assumption on $\tau$ for the construction of the upperbound
in \cite{ahmadi2019backward} is that  $\tau$ has bounded density at the terminal time; in the current setting this is replaced
with the assumption that the random variable ${\bm 1}_{\{\tau \le S\}} Y_\tau^\infty$ has a bounded $\varrho$-moment for some
$\varrho > 1$ (see \eqref{e:Lrhobound}). The other main ingredient in the construction of the upperbound process in
\cite{ahmadi2019backward} is the apriori upperbounds on the supersolution of BSDE; in the current context
this is replaced by the solvability assumption on $S$. Subsection \ref{ss:newxi1} presents a new argument for the terminal
value $\xi_1$ that is completely based on the original BSDE (i.e., it doesn't involve the solution  of an auxiliary linear BSDE).
To simplify arguments this subsection assumes ${\mathbb F}$ to be generated only by the Brownian motion $W$.
The only assumption on $\tau$ is that it be solvable. Let $Y^{\tau,\infty}$ be the supersolution of the BSDE with
terminal condition $\infty$ at terminal time $\tau$. The main idea of this argument is the use of the process $Y^{\tau,\infty}$
as an upperbound to prove \eqref{eq:cont}. Working directly with the original BSDE in constructing upperbounds can lead
to less stringent conditions on model parameters. As an example, we consider in subsection \ref{ss:anexamplexi1} 
the case $S=T$ and $\tau= \inf \{t: |W_t|  = L \}$ which was originally studied in \cite{krus:popi:seze:18} using essentially
a special case of the argument based on the linear auxiliary process which requires the $q$ parameter in assumption \ref{B2}
to satisfy $q >3.$ The new proof given subsection \ref{ss:anexamplexi1} based on the new argument based on solvable stopping times
establishes \eqref{eq:cont} for the minimal supersolution assuming only $ q> 2.$

The argument in Section \ref{sect:term_cond_xi_2}
that proves that the minimal supersolution
corresponding to $\xi_2$ is in fact a solution follows closely the
argument given for the same type of terminal condition in \cite{ahmadi2019backward} for the case $S=T$ deterministic. The assumptions in this section are: 
$S$ is solvable
and ${\mathbb P}(S = \tau ) = 0$;  no solvability
is required for $\tau.$

BSDE with random terminal times are a generalization of elliptic semi-linear PDE
(extension of the Feynman-Kac formula, see \cite{darl:pard:97,pard:99,pard:rasc:14}). 
The works  \cite{dynk:kuzn:98,lega:97,marc:vero:98b,marc:vero:98a} show that the solution of some of these PDE can exhibit a singularity
of the following form on the boundary of the domain $D$
\[
\lim_{x\to \partial D} u(x) = +\infty.
\]
This boundary behavior generalizes to
\[
\lim_{t\to +\infty} Y_{t\wedge \S} =+\infty
\]
for BSDE of the form \eqref{eq:bsde}
(the clearest connection between the last two condition arises when $S$ is a first hitting time of a Markov diffusion process, see subsection \ref{ss:relatedPDE}).
Minimal supersolutions of
BSDE of the type \eqref{eq:bsde}
with $\infty$-valued terminal values at random terminal times can also be used to express the value function of
a class of stochastic optimal control problems over a
random time horizon $[[0,S]]$ with terminal constraints 
of the form ${\bm 1}_A \cdot q_S = 0$, for some $A \in {\mathcal F}_S$, where $q$ is the controlled
process (see \cite{krus:popi:15}).

Strengthening \eqref{eq:term_cond_super_sol} to \eqref{eq:termcondsol} (i.e., going from a supersolution to a solution)
has implications both for BSDE theory and for stochastic optimal control
applications. Consider two distinct terminal values $\xi^1$ and $\xi^2$; with \eqref{eq:term_cond_super_sol} it is impossible to tell
whether the corresponding minimal supersolutions are distinct. Whereas \eqref{eq:termcondsol} guarantees that distinct solutions $Y^1$ and
$Y^2$ correspond to distinct terminal values $\xi^1$ and $\xi^2.$
In stochastic optimal control / finance applications a non-tight optimal control (corresponding to strict inequality in
\eqref{eq:term_cond_super_sol}) can be interpreted as a strictly super-hedging trading strategy. Continuity results overrule such strategies.  
For more comments on these points we refer the reader to \cite{ahmadi2019backward}.

The next two subsections give the definitions, assumptions and results
 we employ from previous works (subsection \ref{ss:integrable}
concerns integrable terminal conditions and subsection \ref{ssect:sing_supersol}
concerns singular terminal values).
The only novelty  is Definition \ref{d:solvable_rtt},
the definition of a solvable stopping time.
We comment on possible future work in the Conclusion (Section \ref{s:conc}).

\subsection{Integrable data}\label{ss:integrable}

Let us start with the definition of solution for BSDE \eqref{eq:bsde}. 
\begin{definition}[Classical solution]  \label{def:sol_BSDE_rtt}
A process $(Y,Z,U,M) = (Y_t,Z_t,U_t,M_t)_{t\geq 0}$, such that 
\begin{itemize}
\item $Y$ is progressively measurable and \cad, 
\item $Z$ is a predictable process with values in $\bR^{d}$, 
\item $M$ is a local martingale orthogonal to $W$ and $\tpi$,
\item $U$ is also predictable and such that for any $t\geq 0$
$$ \int_0^t \int_\cE (|U_s(e)|^2\wedge |U_s(e)|) \mu(de) < +\infty,$$
\end{itemize}
is a solution to the BSDE \eqref{eq:bsde} with random terminal time $\S$ with data $(\xi;f)$ if on the set $\{t\geq \S \}$ $Y_t =\xi$ and $Z_t =U_t=M_t=0$, $\bP$-a.s., $t\mapsto f(t,Y_t,Z_t,U_t) \1_{t\leq T}$ belongs to $L^1_{loc}(0,\infty)$  for any $T\geq 0$, the stochastic integrals w.r.t. $W$ and $\tpi$ are well-defined and, $\bP$-a.s., for all $0\leq t \leq T$,
\begin{eqnarray} \nonumber
Y_{t\wedge \S} & = & Y_{T\wedge \S} + \int_{t\wedge \S}^{T\wedge \S} f(r,Y_r, Z_r,\psi_r) dr -\int_{t\wedge \S}^{T\wedge \S} Z_rdW_r \\ \label{eq:gene_BSDE_rand_time}
& -& \int_{t\wedge \S}^{T\wedge \S} \int_\cE U_r(e) \tpi(de,dr) - \int_{t\wedge \S}^{T\wedge \S} dM_r.
\end{eqnarray}
\end{definition}
For precise definitions on the stochastic integral w.r.t. $\tpi$ and orthogonality, we refer to \cite{jaco:shir:03}. 

In \cite{krus:popi:14,krus:popi:17}, Theorem 3 ensures the existence and uniqueness of a solution, under some conditions on the terminal value $\xi$ and on the generator $f$. Let us evoke them here. 

Firstly the following integrability condition is assumed: for some $r>1$
\begin{equation} \label{eq:int_cond_random_time}
\bE \left[ e^{r\rho \S} |\xi|^r + \int_0^\S e^{r\rho t} |f(t,0,0,{\bm 0})|^r dt \right] < +\infty.
\end{equation}
The constant $\rho$ depends on $r$ and on the generator $f$ (see Remark \ref{rem:def_rho}). We suppose that $f : \Omega \times [0,T] \times \bR \times \bR^{m} \times \mathfrak{B}^2_\mu \to \bR$ is a random measurable function, such that for any $(y,z,\psi)\in  \bR \times \bR^{m} \times \mathfrak{B}^2_\mu$, the process $f(t,y,z,\psi)$ is progressively measurable. For notational convenience we write $f^0_t=f(t,0,0,{\bm 0})$, where
$\mathbf 0$ denotes the null application from $\cE$ to $\bR$. The space $\mathfrak{B}^2_\mu$ is defined\footnote{For the definition of the sum of two Banach spaces, see for example \cite{krei:petu:seme:82}. The introduction of $\mathfrak{B}^2_\mu$ is motivated in \cite{krus:popi:17}.} as follows:
$$\mathfrak{B}^2_\mu =\begin{cases} \bL^2_\mu & \mbox{if } r \geq 2, \\ \bL^1_\mu+\bL^2_\mu & \mbox{if } r < 2, \end{cases}$$
where $\bL^p_\mu=\bL^p(\cE,\mu;\bR)$ is the set of measurable functions $\psi : \cE \to \bR$ such that
$$\| \psi \|^p_{\bL^p_\mu} = \int_{\cE} |\psi(e)|^p \mu(de)  < +\infty.$$

The next conditions are adapted from \cite{krus:popi:15}:
\begin{enumerate}[label=\textbf{(A\arabic*)}]
\item\label{A1} 
The function $y\mapsto f(t,y,z,\psi)$ is continuous and monotone: there exists $\chi \in \bR$ such that a.s. and for any $t \geq 0$ and $z \in \bR^m$ and $\psi \in \mathfrak{B}^2_\mu $
\begin{equation*}
(f(t,y,z,\psi)-f(t,y',z,\psi))(y-y') \leq \chi (y-y')^2.
\end{equation*}
\item \label{A2} For every $j > 0$ and $n \geq 0$, the process 
$$\Upsilon_t(j) = \sup_{|y|\leq j} |f(t,y,0,{\bm 0})-f^0_t|$$
is in $L^1((0,n)\times \Omega)$.
\item \label{A3}  There exists a progressively measurable process $\kappa = \kappa^{y,z,\psi,\phi} : \Omega \times \bR_+\times \bR^m \times \mathfrak{B}^2_\mu \to \bR$ such that
\begin{equation*}
f(t,y,z,\psi)-f(t,y,z,\phi) \leq \int_\cE (\psi(e)-\phi(e))  \kappa^{y,z,\psi,\phi}_t(e)  \mu(de)
\end{equation*}
with $\bP\otimes Leb \otimes \mu$-a.e. for any $(y,z,\psi,\phi)$, $-1 \leq \kappa^{y,z,\psi,\phi}_t(e)$
and $|\kappa^{y,\psi,\phi}_t(e)| \leq \vartheta(e)$ where $\vartheta$ belongs to the dual space of $\mathfrak{B}^2_\mu$, that is $\bL^2_\mu$ or $\bL^\infty_\mu \cap \bL^2_\mu$. 
\item \label{A4}  There exists a constant $L_f$ such that a.s. 
\begin{equation*}
|f(t,y,z,\psi)-f(t,y,z',\psi)| \leq L_z |z-z'|
\end{equation*} for any $(t,y,z,z',\psi)$.
\end{enumerate}
We denote
$$K^2 = \dfrac{1}{2} ( L_z^2+ L_\vartheta^2 ).$$
\begin{Rem}
We can replace {\rm \ref{A3}} by the Lipschitz condition: there exists a constant $L_\vartheta$ such that
\begin{equation*}
|f(t,y,z,\psi)-f(t,y,z,\phi)| \leq L_\vartheta \|\psi - \phi\|_{\mathfrak{B}^2_\mu }.
\end{equation*}
As explained at the beginning of \cite[Section 5]{krus:popi:14}, {\rm \ref{A3}} implies Lipschitz regularity of $f$ w.r.t. $\psi$, with $L_\vartheta$ equal to the norm $\|\vartheta\|_{(\mathfrak{B}^2_\mu)^*}$ of $\vartheta$ in the dual space of $\mathfrak{B}^2_\mu$. However {\rm \ref{A3}} is sufficient to ensure comparison principle for the solution of BSDEs (see \cite[Proposition 5.34]{pard:rasc:14}, \cite[Theorem 3.2.1]{delo:13} or \cite[Remark 4]{krus:popi:14} ). 
\end{Rem}

\begin{Rem} \label{rem:def_rho}
Constant $\rho$ in \eqref{eq:int_cond_random_time} satisfies
\begin{equation} \label{eq:int_cond_random_time_rho}
\rho > \nu = \nu(r) := \begin{cases} \chi + K^2 & \mbox{if } r\geq 2, \\ \chi + \frac{K^2}{r-1}+ \frac{L_\vartheta^2}{\eps(L_\vartheta,r)}& \mbox{if } r < 2.\end{cases}
\end{equation}
where the constant $0 < \eps(L_\vartheta,r) < r-1$ depends only on $L_\vartheta$ and $r$ (see \cite{krus:popi:17}, Section 4). The additional term in $\nu$ disappears if the generator does not depend on the jump part $\psi$ (that is, if $L_\vartheta=0$). 
Even if we can not compute $\eps$ explicitly, we know that 
\begin{equation*} 
0 < \eps \leq (r-1)  \left(2 (\alpha(L_\vartheta,r) + 1)^2-1 \right)^{-\frac{2-r}{2}},
\end{equation*}
and $\alpha(L_\vartheta,r)$ has to be chosen such that for any $x \geq \alpha(L_\vartheta,r)$, 
$$ \frac{1}{2^{r/2}}x^{r} -2^{r/2}- 1 - r(2L_\vartheta+1)x \geq 0. $$
The right-hand side is an increasing function w.r.t. $r \in (1,2)$ and decreasing w.r.t. $L_\vartheta\geq 0$. Hence when $r$ is close to one and $L_\vartheta$ is large, $\eps$ is be very small and thus $\rho$ becomes large. 

\end{Rem}

In \cite{krus:popi:14,krus:popi:17}, a second integrability condition is supposed:
\begin{equation} \label{eq:int_cond_random_time_2}
\bE\left[ \int_0^\S e^{r\rho t}|f(t, e^{-\nu t}\xi_t ,e^{-\nu t}\eta_t, e^{-\nu t}\gamma_t)|^r dt \right] < +\infty,
\end{equation}
where $ \xi_t = \bE (e^{\nu \S} \xi|\cF_t) $ and $( \eta, \gamma, N)$ are given by the martingale representation:
$$e^{\nu \S} \xi = \bE (e^{\nu \S} \xi)  + \int_0^\infty  \eta_s dW_s + \int_0^\infty \int_\cE  \gamma_s(e) \tpi(de,ds) +  N_\S$$
with
\begin{equation*}
\bE \left[ \left( \int_0^\infty | \eta_s|^2 ds +  \int_0^\infty \int_\cE | \gamma_s(e)|^2 \pi(de,ds) +[N]_\S \right)^{r/2} \right] < +\infty.
\end{equation*}

Now \cite[Theorem 3]{krus:popi:14,krus:popi:17} becomes
\begin{theorem} \label{thm:rand_time_exist_sol_BSDE}
Under Conditions {\rm \ref{A1}} to {\rm \ref{A4}} and if $\xi$ and $f^0$ satisfy assumptions \eqref{eq:int_cond_random_time} and \eqref{eq:int_cond_random_time_2} , BSDE \eqref{eq:bsde} has a unique solution $(Y,Z,U,M)$ in the sense of Definition \ref{def:sol_BSDE_rtt} such that for any $0\leq t \leq T$
\begin{eqnarray*} \nonumber
&&\bE \left[ e^{r\rho (t\wedge \S)} |Y_{t\wedge \S}|^r + \int_{0}^{T\wedge \S}  e^{p\rho s} |Y_{s}|^{r}  ds +  \int_{0}^{T\wedge \S}  e^{r\rho s} |Y_{s}|^{r-2} |Z_s|^2 \1_{Y_s\neq 0} ds \right] \\ \nonumber
&& + \bE \left[\int_{0}^{T\wedge \S} e^{r\rho s}  |Y_{s-}|^{r-2}  \1_{Y_{s-}\neq 0} d[ M]^c_s \right] \\  \nonumber
&& + \bE \left[   \int_{t\wedge \S}^{T\wedge \S}\int_\cE e^{r\rho s}  \left( |Y_{s-}|^2 \vee  |Y_{s-} + U_s(e)|^2 \right)^{r/2-1} \1_{|Y_{s-}| \vee  |Y_{s-} + U_s(e)| \neq 0}| U_s(e)|^2\pi(de,ds)\right] \\ \label{eq:rnd_term_time_apriori_estim}
&& + \bE\left[  \sum_{0 < s \leq T\wedge  \S}e^{r\rho s} |\Delta M_s|^2  \left( |Y_{s-}|^2 \vee  |Y_{s-} + \Delta M_s|^2 \right)^{r/2-1} \1_{|Y_{s-}| \vee  |Y_{s-} + \Delta M_s| \neq 0} \right]  
 < +\infty. 
\end{eqnarray*}
And
\begin{eqnarray*} \nonumber
&& \bE \left[ \left( \int_{0}^{\S} e^{2\rho s} |Z_s|^2 ds \right)^{r/2}+ \left(  \int_{0}^{ \S} e^{2\rho s}\int_\cE |U_s(e)|^2 \pi(de,ds)  \right)^{r/2}+  \left( \int_{0}^{ \S} e^{2\rho s} d[M]_s\right)^{r/2} \right] \\ \label{eq:rnd_term_time_apriori_estim_3}
&& \qquad  \leq C \bE\left[  e^{r\rho \S} |\xi|^p  +  \int_{0}^{\S} e^{r\rho s} |f(s,0,0,{\bm 0})|^r ds \right]. 
\end{eqnarray*}
The constant $C$ depends only on $r$, $K$ and $\chi$.
\end{theorem}

In general \eqref{eq:int_cond_random_time_2} is not easy to check. 
Nonetheless if $\xi$ is bounded, we can take $\nu= 0$ in \eqref{eq:int_cond_random_time_2} and assume that:
\begin{equation*}
\bE\left[ \int_0^\S e^{r\rho t}|f(t, \xi_t ,\eta_t, \gamma_t)|^r dt \right] < +\infty,
\end{equation*}
where $ \xi_t = \bE ( \xi|\cF_t) $ and 
$$\xi = \bE ( \xi)  + \int_0^\infty  \eta_s dW_s + \int_0^\infty \int_\cE  \gamma_s(e) \tpi(de,ds) +  N_\S.$$

\subsection{Supersolution for singular terminal condition} \label{ssect:sing_supersol}

To lighten the presentation, {\it in the rest of the paper, $\xi$ is supposed to be non-negative. }Theorem \ref{thm:rand_time_exist_sol_BSDE} gives sufficient conditions to ensure the existence and uniqueness of the solution $(Y,Z,U,M)$. When the terminal condition is singular, that is if $\xi$ does not belong to any $\bL^p(\Omega)$ for some $p>1$, we adopt the following definition. 
\begin{definition}[Supersolution for singular terminal condition] \label{def:sol_sing_BSDE_rtt}
We say that a triple of processes $(Y,Z,U,M)$ is a supersolution to the BSDE \eqref{eq:bsde} with singular terminal condition $Y_\S = \xi \geq 0$ if it satisfies:
\begin{enumerate}
\item There exists some $\ell > 1$ and an increasing sequence of stopping times $\S_n$ converging to $\S$ such that for all $n>0$ and all $t \geq 0$
\begin{align*}
& \bE \left[ \sup_{r\in [0,t]} |Y_{r \wedge \S_n}|^\ell + \left( \int_0^{t \wedge \S_n} |Z_r|^2 dr \right)^{\ell/2} \right. \\
&\qquad  \left. + \left( \int_0^{t \wedge \S_n} \int_\cE |U_r(e)|^2 \pi(de,dr) \right)^{\ell/2}+ [M ]^{\ell/2}_{t \wedge \S_n} \right] < +\infty;
\end{align*}
\item $Y$ is non-negative;
\item for all $0\leq t \leq T$ and $n>0$:
\begin{align} \nonumber
Y_{t \wedge \S_n}  & = Y_{T \wedge \S_n} + \int_{t \wedge \S_n}^{T \wedge \S_n}  f(r,Y_r,Z_r,U_r)dr - \int_{t \wedge \S_\eps}^{T \wedge \S_n} Z_r dW_r  \\ \label{e:sdepart} 
& - \int_{t \wedge \S_n}^{T \wedge \S_n} \int_\cE U_r(e) \tpi(de,dr) - \int_{t \wedge \S_n}^{T \wedge \S_n} dM_r.
\end{align}
\item On the set $\{t\ge \S\}$: $Y_t=\xi, Z=U=M=0$ a.s.\ and \eqref{eq:term_cond_super_sol} holds:
\begin{equation*} 
 \liminf_{t \to +\infty} Y_{t\wedge \S} \geq \xi, \quad \mbox{a.s.}
 \end{equation*}
\end{enumerate}
We say that $(Y,Z,U,M)$ is a minimal supersolution to the BSDE \eqref{eq:bsde} if for any other supersolution $(Y',Z',U',M')$ we have $Y_t\le Y'_t$ a.s.\ for any $t>0$.
\end{definition}

\begin{remark}
The non-negativity condition can be replaced in general by: $Y$ is bounded from below by a process $\bar Y$ such that $\bE  \sup_{t \geq 0} |\bar Y_{t\wedge S}|^\ell < +\infty$.
\end{remark}

We next introduce a concept that we think provides a general and natural framework for the study of
BSDE \eqref{eq:bsde} with singular terminal conditions when the terminal time is a stopping time:
\begin{definition}\label{d:solvable_rtt}
A stopping time $\S$ will be called solvable with respect to the BSDE \eqref{eq:bsde} if the filtration $\bF$ is left-continuous at time $\S$ and 
if the BSDE \eqref{eq:bsde} has a supersolution on the time interval $[\![0,\S]\!]$ with terminal condition $Y_\S = \infty$
that is defined as the limit of the solution of the same BSDE with terminal condition equal to the constant $ k$, as $ k$ tends to $\infty.$
\end{definition}
Most of our arguments will be based on solvable stopping times.
From \cite{krus:popi:15}, we know that every deterministic time $\S$ is solvable provided Conditions {\bf (A)}, \ref{B1} and \ref{B2} below hold. 
Exit times of diffusions from smooth domains provide another
example of a solvable stopping time, see Theorem \ref{thm:main_thm_2} below
(a restatement of \cite[Theorem 2]{krus:popi:15} in terms of solvable times).
\cite[Example 1]{krus:popi:15} shows that any stopping time that has a strictly positive density around $0$ is nonsolvable.
Section \ref{s:solvablelemmas} lists some immediate consequences of the definition above that will be useful in the rest of this article.

\subsubsection{Additional conditions on $f$}

For a singular terminal value $\xi$, the conditions \eqref{eq:int_cond_random_time} and \eqref{eq:int_cond_random_time_2} are false. Hence following \cite{krus:popi:15}, we add some hypotheses concerning the generator $f$ and the terminal random time $\S$. 
\begin{enumerate}[label=\textbf{(B\arabic*)}]
\item\label{B1} 
There exists a constant $q > 1$ and a positive and bounded process $\eta$ such that for any $y \geq 0$
\begin{equation*}
f(t,y,z,\psi)\leq -\frac{y}{\eta_t}|y|^{q-1} + f(t,0,z,\psi).
\end{equation*}
\item \label{B2} The processes $f^0$ and $\xi^-$ are bounded. 
\item \label{B3}  There exists $\delta > \delta^*$ such that $\bE \left[ e^{\delta \S} \right] < +\infty.$ The constant $\delta^*$ depends on $\chi$, $L_z$ and $L_\vartheta$. 
\item \label{B4} There exists $m > m^*$ such that for any $j$
$$\bE \int_0^\S |\Upsilon_t(j)|^m dt <+\infty.$$
The value of $m^*$ depends on $\chi$, $L_z$ and $L_\vartheta$ and also on $\delta$ and $\delta^*$.
\end{enumerate}
We further suppose that the generator $(t,y) \mapsto -y|y|^{q-1}/\eta_t$ satisfies 
the {\bf (A)} and {\bf (B)} assumptions, 
which means that $\eta$ satisfies:
\begin{equation} \label{eq:cond_1_over_eta}
\bE \int_0^T \frac{1}{\eta^m_t} dt < +\infty.
\end{equation}
The values of $\delta^*$ and $m^*$ are given in \cite{krus:popi:15}. Let us simply recall that if $y \mapsto f(t,y,z,\psi)$ is non increasing, that is for $\chi=0$, then we have:
$$\delta^* = 2 K^2,\quad m^* = \dfrac{2\delta }{\delta - 2 K^2 }.$$

We consider $(Y^{(k)},Z^{(k)},\psi^{(k)},M^{(k)})$ the unique solution of the BSDE: for any $t < T$
\begin{align}\nonumber 
Y^{(k)}_{t\wedge \S} & = Y^{(k)}_{T\wedge \S} +\int_{t\wedge \S}^{T\wedge \S}   f(r,Y^{(k)}_r,Z^{(k)}_r,U^{(k)}_r) dr \\ \label{eq:truncated_bsde}
& -\int_{t\wedge \S}^{T\wedge \S}  Z^{(k)}_r dW_r -\int_{t\wedge \S}^{T\wedge \S}   \int_\cE U^{(k)}_r(e) \tpi(de,dr) -\int_{t\wedge \S}^{T\wedge \S} dM^{(k)}_r,
\end{align}
with the truncated terminal condition:
\begin{equation} \label{eq:trunc_term_cond}
\bP-\mbox{a.s., on the set } \{t\geq \S\}, \quad Y^{(k)}_t = \xi \wedge k, \ Z^{(k)}_t = U^{(k)}_t =M^{(k)}_t =0.
\end{equation}
From \cite[Proposition 5]{krus:popi:15}, under {\bf (A)}, {\rm \ref{B2}}, {\rm \ref{B3}} and {\rm \ref{B4}}, there exists a unique solution $(Y^{(k)},Z^{(k)},\psi^{(k)},M^{(k)})$ to the BSDE \eqref{eq:truncated_bsde} and \eqref{eq:trunc_term_cond}.
 
By the comparison principle for BSDEs, the sequence $Y^{(k)}$ is non decreasing and converges to some process $\Ymin$. As for deterministic terminal time, the key point is to obtain an a apriori estimate on $Y^{(k)}$, independent of the constant $k$. This a prior estimate ensures that the stopping time $S$ is solvable in the sense of Definition \ref{d:solvable_rtt}.

\subsubsection{Known results for exit times}

To have such estimate, \cite{krus:popi:15} restricts attention to the case where $\S$ is the first hitting time of a diffusion, namely 
\begin{equation}\label{eq:def_stop_time}
\S = \S_D=\inf\{t\ge 0, \quad \Xi_t \notin D\},
\end{equation}
where the forward process $\Xi$ in $\bR^d$ is the strong solution to the stochastic differential equation
\begin{equation}\label{eq:forwardSDE}
d\Xi_t=b(\Xi_t)dt+\sigma(\Xi_t)dW_t
\end{equation}
with some initial value $\Xi_0 \in \bR^d$. The functions $b:\bR^d\to \bR^d$ and $\sigma:\bR^d \to \bR^{d \times d}$ satisfy a global Lipschitz condition: there exists some $C>0$ such that
\begin{equation}\label{eq:lipschitz_cond_coeff_sde}
\forall x,y\in \bR^d \quad \|\sigma(x)-\sigma(y)\|+\|b(x)-b(y)\| \le C\|x-y\|.
\end{equation}
The domain $D$ is an open bounded subset of $\bR^d$, whose boundary is at least of class $C^2$ (see for example \cite{gilb:trud:01}, Section 6.2, for the definition of a regular boundary). From now on, $\Xi_0$ is fixed and supposed to be in $D$. 

Note that the condition \ref{B3} imposes some implicit hypotheses between the generator $f$, the set $D$ and the coefficients of the SDE \eqref{eq:forwardSDE}. The \cite[Lemma 2]{krus:popi:15} details some sufficient conditions on the coefficients $b$ and $\sigma$. 

We introduce the signed distance function $\dist:\bR^d\to \bR$ of $D$, which is defined by $\dist(x)=\inf_{y \notin D}\|x-y\|$ if $x\in D$ and $\dist(x)=-\inf_{y\in D}\|x-y\|$ if $x\notin D$. \cite[Proposition 6]{krus:popi:15} is a Keller-Osserman type inequality (see \cite{kell:57,osse:57}): there exists a constant $C$ such that:
\begin{equation}\label{eq:upp_bound_rand_time}
0 \leq Y_{t\wedge \S}^{(k)} \le \Ymin_{t\wedge \S} \leq \frac C{ \dist(\Xi_{t\wedge \S})^{2(p-1)}}.
\end{equation}
Constant $p>1$ is the H\"older conjugate of $q$.

Next we define the notion of supersolution. To this end, we set for $n\geq 1$
\begin{equation}\label{eq:def_tau_eps}
\S_n =\inf \left\{t \geq 0, \dist(\Xi_t) \leq \dfrac{1}{n} \right\},
\end{equation}
where $\dist(\Xi_t)$ denotes the distance between the position of $\Xi$ at time $t$ and the boundary of $D$.
The main result \cite[Theorem 2]{krus:popi:15} is:

\begin{theorem}\label{thm:main_thm_2}
If $\S$ is the exit time given by \eqref{eq:def_stop_time}, and if $\bF$ is left-continuous at time $S$, under Assumptions \emph{\textbf{(A)}} and \emph{\textbf{(B)}}, $\S$ is a solvable stopping time (Definition \ref{d:solvable_rtt}). Moreover there exists a minimal supersolution $(\Ymin,\Zmin,\Pmin,\Mmin)$ to BSDE \eqref{eq:bsde} with singular terminal condition $\Ymin_\S=\xi$ (Definition \ref{def:sol_sing_BSDE_rtt}).
\end{theorem}
Let us emphasize that estimate \eqref{eq:upp_bound_rand_time} implies that a.s. $\Ymin_t \leq C n^{2(p-1)}$ if $t \leq \S_n$. This property is similar to the result in Lemma \ref{lem:sq_stop_times_bounded_Y} in the continuous case. 

\section{Solvable stopping time and minimal supersolution}\label{s:solvablelemmas}

The next lemmas are useful consequences of the notion of solvable stopping times. First, note that the left-continuity assumption of $\bF$ at time $\S$ is true for example if $\S$ is predictable and if $\bF$ is a quasi-left continuous filtration (that is for any predictable stopping time $\tau$, we have $\cF_{\tau-}=\cF_{ \tau}$). This property of the filtration rules out the possibility that any of the involved processes has jumps at predictable, and a fortiori deterministic times. An important example is the filtration generated by the Brownian motion $W$ and the orthogonal Poisson random measure $\pi$ and $\S$ is given by \eqref{eq:def_stop_time}. 

\begin{lemma} \label{lem:solvability_min_sol}
Assume that $\S$ is solvable and suppose that the generator $f$ satisfies Conditions {\bf (A)}. Then the BSDE  \eqref{eq:bsde} has a minimal supersolution on the time interval $[\![0,\S]\!]$ with terminal condition $Y_\S = \infty$. 
\end{lemma}
\begin{proof}
The arguments can be found in \cite[Propositions 4 and 7]{krus:popi:15}. The adaptation is straightforward in our setting since the arguments are not based on a particular form of the stopping time $\S$. Only left-continuity of the filtration is important. 
\end{proof}
Let us emphasize that Assumptions {\bf (B)} are not necessary here, since solvability implies existence of a supersolution. 
In the rest of the paper we denote by $(Y^{\infty},Z^{\infty},U^{\infty},M^{\infty})$ the minimal weak supersolution with terminal condition $+\infty$ a.s. at time $\S$. Sometimes, if we want to stress the dependence w.r.t. $\S$, we denote it $(Y^{\S,\infty},Z^{\S,\infty},U^{\S,\infty},M^{\S,\infty})$.

\begin{lemma}
Assume that $\S$ is solvable and suppose that generator $f$ satisfies Conditions {\bf (A)}, {\rm \ref{B2}}, {\rm \ref{B3}} and {\rm \ref{B4}}. Then the BSDE  \eqref{eq:bsde} with a singular Markovian terminal value $\xi$ at time $\S$, has a minimal supersolution $(\Ymin,\Zmin,\Umin,\Mmin)$ on the time interval $[\![0,\S]\!]$ with terminal condition $\Ymin_\S = \xi$. 
\end{lemma}
\begin{proof}
Let us denote by $ Y^{(k),\infty}$ the first component of the solution of the BSDE \eqref{eq:bsde} with terminal condition $k$. Since $\S$ is solvable, and with {\bf (A)}, $Y^{(k),\infty}$ is an increasing sequence converging to $Y^\infty$. 

Again from \cite[Proposition 5]{krus:popi:15}, under {\bf (A)}, {\rm \ref{B2}}, {\rm \ref{B3}} and {\rm \ref{B4}}, there exists a unique solution $(Y^{(k)},Z^{(k)},\psi^{(k)},M^{(k)})$ to the BSDE \eqref{eq:truncated_bsde} and \eqref{eq:trunc_term_cond}. By comparison principle, a.s for any $t \geq 0$
$$Y^{(k)}_t \leq Y^{(k),\infty}_t \leq Y^\infty_t.$$
Hence we obtain an upper estimate on $Y^{(k)}$, independent of $k$, which replaces the upper bound \eqref{eq:upp_bound_rand_time}. Arguing now as in \cite{krus:popi:15}, we obtain the existence of $(\Ymin,\Zmin,\Umin,\Mmin)$. 
\end{proof}

Note that the main result of Theorem \ref{thm:main_thm_2} is the solvability of the first exit time $\S$. The existence of $(\Ymin,\Zmin,\Pmin,\Mmin)$ comes from the preceding lemma.

Before we move further, let us note the following:
\begin{lemma} \label{lem:sq_stop_times_bounded_Y}
Suppose a stopping time $S$ is solvable. Suppose $(Y,Z,U,M)$ is a supersolution of \eqref{eq:bsde} with terminal
condition $\xi$ constructed as the limit of solutions with terminal condition $\xi \wedge k.$ Then
the sequence $S_n$ in Definition \ref{def:sol_sing_BSDE_rtt} can be chosen so that
\begin{equation}\label{e:boundedatbetan}
Y_t \le n \text{ for }t < S_n.
\end{equation}
\end{lemma}
\begin{proof}
Let $Y^{S,\infty}$ denote the first component of the supersolution for terminal condition $\infty$ and
let $S_n^{1,\infty}$ be the sequence of $S_n$ in Definition \ref{def:sol_sing_BSDE_rtt} for the same terminal condition.
It follows from \eqref{e:sdepart} and \eqref{eq:term_cond_super_sol} that $Y^{S,\infty}$ has c\`adl\`ag sample paths
on $[\![0,S]\!]$ and $\lim_{t\rightarrow \infty}Y^{S,\infty}_{t \wedge S} = \infty.$
This implies that the hitting times
\begin{equation}\label{d:betan2}
S_n^{2,\infty} \doteq \inf\{t: Y^{S,\infty}_{t \wedge S} \ge n \}
\end{equation}
satisfy: $S_n^{2,\infty} \leq S$ and it is a non-decreasing sequence. From the first property of a supersolution, this sequence converges almost surely to $S$. Now suppose that $S_N^{2,\infty} = S$ for some $N$ (and thus for any $n\geq N$). It would mean that $Y^{S,\infty}$ has a jump at time $S$. In other words, the martingale parts have a jump at time $S$. But it is excluded in Definition \ref{d:solvable_rtt}. Thus 
\begin{equation}\label{e:limitofbetan2}
S_n^{2,\infty} \nearrow S \text{ as } n \nearrow \infty. 
\end{equation}
Then if we replace the stopping times $S^{1,\infty}_n$
in Definition \ref{def:sol_sing_BSDE_rtt} with 
$S^{3,\infty}_n\doteq S^{1,\infty}_n \wedge S^{2,\infty}_n$ 
all of the conditions of the definition remain valid; furthermore 
\begin{equation}\label{e:boundonYinf}
Y_t^{S,\infty} \le n \text{ for }t < S^{3,\infty}_n,
\end{equation}
holds. This proves the lemma for the terminal condition $\infty.$
Let $Y^{S,k}$ denote the solution of \eqref{eq:bsde} with terminal condition $Y_S = k.$
Then by definition $Y^{S,k}_{t\wedge S} \nearrow Y^{S,\infty}_{t\wedge S}$. This and \eqref{e:boundonYinf} imply
\begin{equation}\label{e:orderLinf}
Y_t^{S,k} \le
Y_t^{S,\infty} \le n \text{ for }t < S^{3,\infty}_n.
\end{equation}
Let $Y^{S,\xi}$ be the minimal supersolution of \eqref{eq:bsde} with terminal condition $Y_S = \xi$
and let $Y^{S,\xi\wedge k}$ be the solution of \eqref{eq:bsde} with terminal condition $Y_S = \xi \wedge k.$
By the assumption of the lemma
\begin{equation}\label{e:limitinL}
Y^{S,\xi \wedge k }_{t\wedge S} \nearrow Y^{S,\xi}_{t\wedge S}
\end{equation}
as $k \nearrow \infty$. 
By comparison principle for the solution of BSDE we have 
$Y^{S,\xi \wedge k}_{t\wedge S} \le  Y^{S,k}_{t\wedge S}$. This, \eqref{e:orderLinf},
\eqref{e:limitinL}, the definition \eqref{d:betan2} of $S^{2,\infty}_n$ and letting $k \nearrow \infty$ imply
\begin{equation}\label{e:upperboundxibeta2}
Y_t^{S,\xi} \le
Y_t^{S,\infty} \le n \text{ for }t < S^{3,\infty}_n.
\end{equation}
Let $S_n^{1,\xi}$ be the sequence of stopping time appearing in the definition of the supersolution
$Y^{S,\xi}.$ Define $S_n^{2,\xi} \doteq S_n^{1,\xi} \wedge S_n^{3,\infty}.$ From \eqref{e:boundedatbetan}
and from the assumption that $\beta_n^{1,\xi} \nearrow \beta$  we infer $S_n^{2,\xi} \nearrow S.$
This implies that if we replace the replace $S_n^{1,\xi}$ with $S_n^{2,\xi}$ all of the conditions
appearing in the definition of the supersolution $Y^{S,\xi}$ continue to hold; by \eqref{e:upperboundxibeta2}
this sequence of stopping times also satisfy
\begin{equation}\label{e:orderLinf2}
Y_t^{S,\xi} \le
Y_t^{S,\infty} \le n \text{ for }t < S^{2,\xi}_n.
\end{equation}
This proves the lemma for the terminal condition $\xi.$
\end{proof}

If we work with the filtration $\bF^W$ generated by the Brownian motion $W$, then BSDE \eqref{eq:bsde} reduces to the following:
\begin{equation}\label{eq:bsdebrownian}
dY_t  = - f(t,Y_t,Z_t) dt + Z_t dW_t.
\end{equation}
\begin{corollary} \label{coro:seq_rtt_bounded_Y}
In the Brownian filtration $\bF^W$, if $S$ is solvable, then \eqref{e:boundedatbetan} becomes: 
\begin{equation}\label{e:boundedatbetan2}
Y_t \le n \text{ for }t \leq S_n.
\end{equation}
\end{corollary}
\begin{proof}
Indeed the trajectories of $Y$ are now continuous, not only c\`adl\`ag. 
\end{proof}

\section{On the existence of a limit} \label{sect:existence_limit}

In Definition \ref{def:sol_sing_BSDE_rtt}, we only supposed that \eqref{eq:term_cond_super_sol} holds: a.s. 
$$ \liminf_{t \to +\infty} \Ymin_{t\wedge \S} \geq \xi.$$
If $\xi=+\infty$ a.s. then we immediately obtain that 
$$ \liminf_{t \to +\infty} Y^\infty_{t\wedge \S} = \lim_{t \to +\infty} Y^\infty_{t\wedge \S} = +\infty.$$
In this section, we focus on the existence of the limit, that is, does it hold that a.s. 
$$ \liminf_{t \to +\infty} \Ymin_{t\wedge \S}  =  \lim_{t \to +\infty} \Ymin_{t\wedge \S} \ ?$$
This question was studied in \cite{popi:16} for a deterministic final time $T$ and the result remains true in our setting. 

We suppose that $\S$ is a solvable stopping time and that Conditions {\bf(A)} and {\bf (B)} hold. Hence for any $\xi$, we can consider the minimal supersolution $(\Ymin,\Zmin,\Umin,\Mmin)$ of BSDE \eqref{eq:bsde} with terminal condition $\xi$ at time $\S$, which is obtained as the increasing limit of the solution with terminal condition $\xi \wedge k$. 

Roughly speaking, the limit of $\Ymin_{\cdot\wedge \S}$ exists provided we know the precise behavior of the generator $f$ w.r.t. $y$. The details can be found in \cite{popi:16} and are left to the reader. We break the generator $f$ into four parts:
\begin{align} \nonumber
f(s,y,z,\psi) & = \left[f(s,y,z,\psi)-f(s,0,z,\psi)\right] +  \left[f(s,0,z,\psi)-f(s,0,0,\psi)\right] \\ \nonumber
& +  \left[f(s,0,0,\psi)-f(s,0,0,{\bm 0})\right] +  f^0_s \\ \label{eq:decomp_gene}
& = \phi(s,y,z,u) + \varpi(s,z,\psi) + \varrho(s,\psi) + f^0_s.
\end{align}
Moreover we suppose that 
\begin{enumerate}[label=\textbf{(C\arabic*)}]
\item\label{C1} The generator $f$ satisfies
\begin{equation*}
b_t g(y) \leq f(t,y,z,\psi)-f(t,0,z,\psi) , \quad \forall y \geq 0, \ \forall (t,z,\psi),
\end{equation*}
where 
\begin{itemize}
\item $b$ is positive and $\displaystyle \bE \int_0^{\S} b_s ds < +\infty$;
\item $g$ is a negative, decreasing and of class $C^1$ function and concave on $\bR_+$ with $g(0)<0$ and $g'(0)<0$.  
\end{itemize}
\item\label{C2}  Moreover one of the next three cases holds:
\begin{itemize}
\item {\bf Case 1.} $f$ does not depend on $\psi$ or $\varrho(t,\psi) \geq 0$;
\item {\bf Case 2.} The value $\vartheta$ of \ref{A3} belongs to $\bL^1_\mu(\cE)$ and there exists a constant $\kappa_*>-1$ such that $\kappa^{0,0,\psi,0}_s(e) \geq \kappa_*$ a.e. for any $(s,\psi,e)$;
\item {\bf Case 3.} $\mu$ is a finite measure on $\cE$.
\end{itemize}

\end{enumerate}
Since Conditions \textbf{(B)} should hold, in particular \ref{B1}, we deduce that $b_t g(y) \leq -\dfrac{1}{\eta_t} y|y|^{q-1}$ for any $t \geq 0$ and $y$. Thus w.l.o.g. $g(y) \leq - y|y|^q$ and $ b_t \geq (-1/g(1))\dfrac{1}{\eta_t} = C \dfrac{1}{\eta_t}$ for some positive constant $C$. We can always add to $g$ a linear function like $- y-1$ such that $g(0)< 0$ and $g'(0) < 0$. Let us define the function $\Theta$ on $(0,+\infty)$ by
\begin{equation}
\Theta (x) = \int_{x}^{+\infty} \frac{-1}{g(y)} dy .
\end{equation}
Recall that $g$ is continuous and negative on $\bR_+$. Thus from the condition $g(y) \leq -y|y|^q$, the function $\Theta : [0,+\infty) \to (0,\Theta(0)]$ is well defined, decreasing, of class $C^{1}$, and bijective. Let $\Theta^{-1} : (0,\Theta(0)] \to [0,+\infty)$ be the inverse of $\Theta$. 

The next theorem shows that process $\Ymin$ is c\`adl\`ag on $\bR_+$ when filtration $\bF$ is complete and right-continuous. No additional assumption (left-continuity) on the filtration is needed here. 
\begin{theorem} \label{thm:exists_limit}
Assumptions {\rm \textbf{(A)}}, {\rm \textbf{(B)}} and {\rm \textbf{(C)}} hold. Then the minimal supersolution $\Ymin$ is equal to: a.s. for any $t\geq 0$
$$\Ymin_{t\wedge\S} = \Theta^{-1} \left( \bE \left[ \Theta(\xi) - \Phi^+_{t\wedge\S} + \Phi^-_{t\wedge\S}  \bigg| \cF_t \right] \right).$$
The processes $\Phi^+$ and $\Phi^-$ are two non-negative c\`adl\`ag  supermartingales with a.s. $\displaystyle \lim_{t\to +\infty} \Phi^-_{t\wedge\S} = 0$. 
\end{theorem}
Now $\Phi^+$ being a non-negative c\`adl\`ag supermartingale, we can deduce the existence of the following limit:
$$\lim_{t\to +\infty} \Phi_{t \wedge \S}^+ :=\Phi_{ \S-}^+$$
Thereby the limit of $\Ymin$ exists
$$\lim_{t\to +\infty} \Ymin_{t \wedge \S}= \Theta^{-1}\left( \Theta(\xi) -\Phi_{ \S-}^+ \right) \geq \xi.$$
In other words, $\Ymin$ is a c\`adl\`ag process. 
\begin{proof}
We follow the arguments developed in the proof of \cite[Lemma 2.3]{popi:16}. We only have to handle the stopping time $\S$. 
Since $Y^{(k)}_t$ is bounded from below by zero, we can apply It\^o's formula: for $0 \leq t \leq T$
\begin{eqnarray} \nonumber
&& \Theta(Y^{(k)}_{t\wedge \S}) = \Theta(Y^{(k)}_{T\wedge \S}) + \int_{t\wedge \S}^{T\wedge \S} \Theta'(Y^{(k)}_{s-}) f(s,Y^{(k)}_s,Z^{(k)}_s,U^{(k)}_s) ds \\ \nonumber
&&\quad  -  \int_{t\wedge \S}^{T\wedge \S} \Theta'(Y^{(k)}_{s-}) Z^{(k)}_s dW_s - \int_{t\wedge \S}^{T\wedge \S} \Theta'(Y^{(k)}_{s-})  \int_{\cE} U^{(k)}_s(e) \tpi(de,ds)- \int_{t\wedge \S}^{T\wedge \S}  \Theta'(Y^{(k)}_{s-})  dM^{(k)}_s \\ \nonumber
&&\quad  - \frac{1}{2} \int_{t\wedge \S}^{T\wedge \S} \Theta''(Y^{(k)}_{s-}) |Z^{(k)}_s|^2 ds - \frac{1}{2} \int_{t\wedge \S}^{T\wedge \S} \Theta''(Y^{(k)}_{s-}) d[M^{(k)}]^c_s \\ \nonumber
&&\quad  -  \int_{t\wedge \S}^{T\wedge \S} \int_\cE \left[\Theta(Y^{(k)}_{s-} + U^{(k)}_s(e))-\Theta(Y^{(k)}_{s-})-\Theta'(Y^{(k)}_{s-})U^{(k)}_s(e)\right]\pi(ds,de) \\ \nonumber
&&\quad  -  \sum_{t\wedge \S <s \leq T\wedge \S} \left[\Theta(Y^{(k)}_{s-}+ \Delta M^{(k)}_s)-\Theta(Y^{(k)}_{s-})-\Theta'(Y^{(k)}_{s-})\Delta M^{(k)}_s \right] \\  \label{eq:Ito_form_Theta}  
&& = \bE^{\cF_t} \Theta(Y^{(k)}_{T\wedge \S}) - \Phi^{(k)}_{t\wedge \S,T \wedge \S}
\end{eqnarray}
where 
\begin{align*}
\Phi^{(k)}_{t\wedge \S,T \wedge \S} & = -\bE^{\cF_t}  \int_{t\wedge \S}^{T\wedge \S} \Theta'(Y^{(k)}_{s-}) f(s,Y^{(k)}_s,Z^{(k)}_s,U^{(k)}_s) ds +\frac{1}{2} \bE^{\cF_t} \int_{t\wedge \S}^{T\wedge \S} \Theta''(Y^{(k)}_{s-}) |Z^{(k)}_s|^2 ds \\
& + \frac{1}{2}\bE^{\cF_t}  \int_{t\wedge \S}^{T\wedge \S} \Theta''(Y^{(k)}_{s-}) d[M^{(k)}]^c_s  \\
& + \bE^{\cF_t}  \sum_{t\wedge \S <s \leq T\wedge \S} \left[\Theta(Y^{(k)}_{s-}+ \Delta M^{(k)}_s)-\Theta(Y^{(k)}_{s-})-\Theta'(Y^{(k)}_{s-})\Delta M^{(k)}_s \right]\\
&+  \bE^{\cF_t}\int_{t\wedge \S}^{T\wedge \S} \int_\cE \left[\Theta(Y^{(k)}_{s-} + U^{(k)}_s(e))-\Theta(Y^{(k)}_{s-})-\Theta'(Y^{(k)}_{s-})U^{(k)}_s(e)\right]\pi(ds,de) .
\end{align*}
We use the decomposition \eqref{eq:decomp_gene} of the generator $f$. Since $\Theta$ is non increasing and convex, the next terms are non-negative:
\begin{align*}
& \bE^{\cF_t}  \sum_{t\wedge \S <s \leq T\wedge \S} \left[\Theta(Y^{(k)}_{s-}+ \Delta M^{(k)}_s)-\Theta(Y^{(k)}_{s-})-\Theta'(Y^{(k)}_{s-})\Delta M^{(k)}_s \right] \\
&  \frac{1}{2}\bE^{\cF_t}  \int_{t\wedge \S}^{T\wedge \S} \Theta''(Y^{(k)}_{s-}) d[M^{(k)}]^c_s \\
& -\bE^{\cF_t}  \int_{t\wedge \S}^{T\wedge \S} \Theta'(Y^{(k)}_{s-})  f^0_s ds
\end{align*}
and we can use monotone convergence theorem to pass to the limit as $T$ tends to $+\infty$. 

Starting for the inequality: $ \varpi(s,z,\psi)  \geq - L_z |z|$, and using the concavity of $g$, we obtain that 
$$- \Theta'(Y^{(k)}_{s-}) \varpi(s,Z^{(k)}_s,U^{(k)}_s) ds +\frac{1}{2} \Theta''(Y^{(k)}_{s-}) |Z^{(k)}_s|^2 ds \geq \dfrac{L_z^2}{2g'(Y^{(k)}_{s})} \geq  \dfrac{L_z^2}{2g'(0)} .$$
And
$$ - \Theta'(Y^{(k)}_{s-}) \phi(s,Y^{(k)}_s,Z^{(k)}_s,U^{(k)}_s) \geq - b_s.$$
Since $\bE \int_0^\S b_s < +\infty$ and from \ref{B3}, we deduce that the negative part of
$$- \Theta'(Y^{(k)}_{s-}) \left[ f(s,Y^{(k)}_s,Z^{(k)}_s,U^{(k)}_s) -f(s,0,0,U^{(k)}_s)\right] ds +\frac{1}{2}\Theta''(Y^{(k)}_{s-}) |Z^{(k)}_s|^2 ds $$
is bounded in $L^1$, uniformly w.r.t. $(T,k)$. The remaining term is
\begin{align*}
& -\bE^{\cF_t}  \int_{t\wedge \S}^{T\wedge \S} \Theta'(Y^{(k)}_{s-}) \left[ f(s,0,0,U^{(k)}_s) -f^0_s \right] ds  \\
& +  \bE^{\cF_t}\int_{t\wedge \S}^{T\wedge \S} \int_\cE \left[\Theta(Y^{(k)}_{s-} + U^{(k)}_s(e))-\Theta(Y^{(k)}_{s-})-\Theta'(Y^{(k)}_{s-})U^{(k)}_s(e)\right]\pi(ds,de) . 
\end{align*}

Assume that $f$ does not depend on $\psi$ or that $\varrho(s,\psi) \geq 0$ ({\bf Case 1}). Again from the convexity of $\Theta$, this last term is non-negative. Our previous arguments show that we can pass to the limit when $T$ goes to $+\infty$ in \eqref{eq:Ito_form_Theta}:
$$ \Theta(Y^{(k)}_{t\wedge \S})  =  \bE^{\cF_t} \Theta( \xi \wedge  k ) - \Phi^{(k)}_{t\wedge \S,\S}.$$
Then by monotone convergence theorem, we obtain the convergence (in $\bL^1$) of $\Phi^{(k)}_{t\wedge \S,\S}$ to some process $\Phi_t$ and:
\begin{equation} \label{eq:explicit_expr_Y}
\Theta(Y_{t\wedge \S})= \bE^{\cF_t}[\Theta(\xi)] -\Phi_{t\wedge \S} .
\end{equation}
We can decompose the process $\Phi$:
$$\Phi_{t\wedge \S}  = \Phi^+_{t\wedge \S}  - \Phi^-_{t\wedge \S} ,$$
such that $ \Phi^+ $ and $ \Phi^-$ are non-negative c\`adl\`ag supermartingales with:
$$ \Phi^-_{t\wedge \S}  \leq \bE^{\cF_t} \int_{t\wedge \S}^{\S} \left( b_s - \dfrac{L_z^2}{2g'(0)} \right) ds.$$
In particular a.s. 
$$\lim_{t\to+\infty} \Phi^-_{t\wedge \S} = 0.$$
For the {\bf Case 2} and the {\bf Case 3}, we can exactly use the same arguments as in \cite{popi:16}. We skip them here. 
This achieves the proof of the theorem.
\end{proof}

\begin{remark}
A careful reading shows that {\rm \ref{B1}} is unnecessary. We only need that the function $\Theta$ is well-defined. 
\end{remark}

\section{Markovian terminal conditions} \label{sect:markov_setting}

In this section, we assume that Conditions {\bf (A)} and {\bf (B)} hold and that $\S$ is given by \eqref{eq:def_stop_time}. Thereby $\S$ is a solvable stopping time (Theorem \ref{thm:main_thm_2}). We further suppose that 
\begin{enumerate}[label=\textbf{(D\arabic*)}]
\item\label{D1} The terminal data $\xi$ satisfies
$$\xi = g(\Xi_{\S}),$$
where $g : \bR^{d} \to \overline{\bR_{+}}$ is a function such that $F_{\infty} = \left\{ g = + \infty \right\} \cap \partial D$ is a closed set.
\item\label{D2} On $\bR^{d} \setminus F_{\infty}$, $g$ is locally bounded, that is, for all compact set $\mathcal{K} \subset \bR^{d} \setminus F_{\infty}$, 
$$g \mathbf 1_{\mathcal{K}} \in L^{\infty}(\bR^{d}).$$
\item \label{D3} The boundary $\partial D$ belongs to $C^{3}$.
\end{enumerate}

To obtain the continuity, we start with a technical result. We know that estimate \eqref{eq:upp_bound_rand_time} holds:
\begin{equation*}
0 \leq Y_{t\wedge \S}^{(k)} \le \Ymin_{t\wedge \S} \leq \frac C{ \dist(\Xi_{t\wedge \S})^{2(p-1)}}.
\end{equation*}
The constant $C$ depends on $q$, $D$ and the bound on $b$ and $\sigma$. Here we construct another estimate which depends also on the function $g$.  
\begin{lemma} \label{lem:major_Ymin_with_g}
If $U$ is an open set such that $\overline{U} \cap F_{\infty} = \emptyset$ and $U \cap \partial D \neq \emptyset$, then there exists a constant $C = C(U,g,q,b,\sigma,D)$ and an open set $D_{U}$ such that $D \subset D_{U}$ and if $\dist_{U}$ denotes the distance to the boundary of $D_{U}$, we have
\begin{equation} \label{majortech}
\bP - \mbox{a.s.} \quad \forall k \in \bN, \ \forall t \geq 0, \ Y^{(k)}_{t} \leq \frac{C}{\left( \dist_{U}(\Xi_{t \wedge \S}) \right)^{2(p-1)}}.
\end{equation}
Recall that $\S$ is always the first exit time from $\overline{D}$.
\end{lemma}
The proof is a straightforward adaptation of \cite[Proposition 7]{popi:07} and \cite[Proposition 6]{krus:popi:15}. The second technical result concerns $(\Zmin,\Umin)$, and it is the extension of \cite[Propositions 4 and 8]{popi:07} (a similar result was not proven in \cite{krus:popi:15}). 
\begin{lemma} \label{lem:major_Zmin_with_g}
Under assumptions {\rm {\bf (A)}} and {\rm {\bf (B)}}, for any $\eps > 1$, there exists a constant $C$ such that 
$$ \bE \int_{0}^{\S} \left( \| \Zmin_{r} \|^{2} +  \int_{\cE} \left| \Umin_r(e) \right|^2 \mu(de)  \right) \dist(\Xi_{r})^{4(p-1)+\eps} dr \leq C.$$
This inequality holds if we replace $\Zmin$ and $\Umin$ by $Z^{(k)}$ and $U^{(k)}$. If Condition {\rm {\bf (D)}} holds, then we can replace $\dist$ by $\dist_U$, with a modification of the value of the constant $C$. 
\end{lemma}
\begin{proof}
The beginning of the proof is similar to \cite[Proposition 6]{krus:popi:15}. Let $\lambda>0$ and introduce the set $D_\lambda=\{x\in \bR^d , \ |\dist(x)|\le \lambda\}$. Then it follows from Lemma 14.16 in \cite{gilb:trud:01} that there exists a positive constant $\lambda_0$ such that $\dist \in C^2(D_{\lambda_0})$. Since $D$ is bounded there exists a constant $R>0$ such that $0\le \dist(x)\le R$ for all $x\in \overline D$. Let $\varphi\in C^\infty(\bR^d,[0,1])$ with $\varphi=1$ on $\bR^d\setminus D_{\lambda_0}$ and $\varphi=0$ on $D_{\lambda_0/2}$. We define a function $\zeta \in C^2(\bR^d,\bR_+)$ such that $\zeta=(1-\varphi)\dist+R\varphi$ on $\overline D$. Since $\zeta \ge \dist \geq 0$ on $\overline D$, $x \mapsto |\zeta(x)|^{4(p-1) + \eps}$ is not in $C^{2} (\bR^{d})$, but this function belongs to $C^{2}(D \setminus D_{\lambda_0})$ and we can define this function on the rest of $(\bR^{d} \setminus D) \cup D_{\lambda_0}$ in order to have the required regularity. For $\lambda < \lambda_0$, define 
$$S_\lambda = \inf \{ t\geq 0, \ \Xi_t \in D_\lambda\}.$$
Take $\lambda$ sufficiently small such that $\Xi_0 \in D \setminus D_{\lambda}$. 
The It\^{o} formula leads to:
\begin{align} \nonumber
& \left( Y^{(k)}_{t \wedge \S_{\lambda}} \right)^{2} \zeta(\Xi_{t \wedge \S_{\lambda}})^{4(p-1)+\eps} = \left( Y^{(k)}_{0} \right)^{2} \zeta(\Xi_{0})^{4(p-1)+\eps}  + \int_{0}^{t \wedge \S_{\lambda}} \| Z^{(k)}_{r} \|^{2} \zeta (\Xi_{r})^{4(p-1)+\eps} dr \\ \nonumber
& \quad  - 2 \int_{0}^{t \wedge \S_{\lambda}} Y^{(k)}_{r} f(r,Y^{(k)}_{r},Z^{(k)}_{r},U^{(k)}_{r}) \zeta(\Xi_{r})^{4(p-1)+\eps} dr \\ \nonumber
& \quad +  2 \int_{0}^{t \wedge \S_{\lambda}} Y^{(k)}_{r} \zeta(\Xi_{r})^{4(p-1)+\eps}\left(  Z^{(k)}_{r} dW_{r} + \int_\cE U^{(k)}_r(e) \tpi(de,dr) + dM^{(k)}_r \right) \\ \nonumber
& \quad + \int_{0}^{t \wedge \S_{\lambda}} \zeta (\Xi_{r})^{4(p-1)+\eps}  \int_{\cE} \left| U^{(k)}_r(e) \right|^2 \pi(de,dr) \\ \nonumber
& \quad + \int_{0}^{t \wedge \S_{\lambda}} \zeta (\Xi_{r})^{4(p-1)+\eps}  d[M^{(k)}]^c_r + \sum_{0< s \leq  t \wedge \S_{\lambda}}  \zeta (\Xi_{r})^{4(p-1)+\eps} (\Delta M^{(k)}_r )^2 \\ \nonumber
& \quad + (4(p-1)+\eps) \int_{0}^{t \wedge \S_{\lambda}} (Y^{(k)}_{r})^{2} \zeta(\Xi_{r})^{4(p-1)+\eps -1}\nabla \zeta(\Xi_{r}) \left( b(\Xi_{r}) dr + \sigma(\Xi_{r}) dW_{r} \right) \\ \nonumber
& \quad + \frac{ (4(p-1)+\eps)}{2} \int_{0}^{t \wedge \S_{\lambda}} (Y^{(k)}_{r})^{2} \left[ (4(p-1)+\eps -1) \zeta(\Xi_{r})^{4(p-1)+\eps -2} \|\sigma(\Xi_{r}) \nabla \zeta (\Xi_{r}) \|^{2} \right.\\ \nonumber
& \qquad \qquad \qquad \qquad \left.+ \zeta(\Xi_{r})^{4(p-1)+\eps -1} \tr \left( \sigma \sigma^{*} (\Xi_{r}) D^{2} \zeta(\Xi_{r}) \right) \right] dr \\ \label{eq:Ito_form_control_Z_psi}
& \quad +  2(4(p-1)+\eps) \int_{0}^{t \wedge \S_{\lambda}} Y^{(k)}_{r} \zeta(\Xi_{r})^{4(p-1)+\eps -1} Z^{(k)}_{r} \nabla \zeta(\Xi_{r}) \sigma(\Xi_{r}) dr.
\end{align} 
Taking the expectation removes all martingale terms. 
From \eqref{eq:upp_bound_rand_time}, we know that there exists a constant such that for any $k$ and all $ t\geq 0$, 
$$\left( Y^{(k)}_{t \wedge \S_{\lambda}} \right)^{2} \zeta(\Xi_{t \wedge \S_{\lambda}})^{4(p-1)} \leq C.$$
Thereby the terms 
\begin{align*}
&  (4(p-1)+\eps) \int_{0}^{t \wedge \S_{\lambda}} (Y^{(k)}_{r})^{2} \zeta(\Xi_{r})^{4(p-1)+\eps -1}\nabla \zeta(\Xi_{r}) b(\Xi_{r}) dr\\
& \quad + \frac{ (4(p-1)+\eps)}{2} \int_{0}^{t \wedge \S_{\lambda}} (Y^{(k)}_{r})^{2} \left[ (4(p-1)+\eps -1) \zeta(\Xi_{r})^{4(p-1)+\eps -2} \|\sigma(\Xi_{r}) \nabla \zeta (\Xi_{r}) \|^{2} \right.\\
& \qquad \qquad \qquad \qquad \left.+ \zeta(\Xi_{r})^{4(p-1)+\eps -1} \tr \left( \sigma \sigma^{*} (\Xi_{r}) D^{2} \zeta(\Xi_{r}) \right) \right] dr
\end{align*} 
are bounded by
$$C \left( \bE \int_{0}^{\S} \zeta^{\eps-1}(\Xi_{r}) dr +\bE \int_{0}^{\S} \zeta^{\eps-2}(\Xi_{r}) dr \right).$$
For $\eps > 1$, the arguments developed in the proof of \cite[Proposition 4]{popi:07} show that these integrals are finite. The Cauchy--Schwarz inequality leads to
\begin{eqnarray*}
&&\left| \bE \int_{0}^{t \wedge \S_{\lambda}} Y^{(k)}_{r} \zeta(\Xi_{r})^{4(p-1)+\eps -1} Z^{(k)}_{r} \nabla \zeta(\Xi_{r}) 
\sigma(\Xi_{r}) dr \right| \\
&& \quad \leq  \left( \bE \int_{0}^{t \wedge \S_{\lambda}} \|Z^{(k)}_{r} \|^{2} \zeta(\Xi_{r})^{4(p-1)+\eps} dr \right)^{1/2} \\
&& \quad \qquad  \times \left( \bE \int_{0}^{t \wedge \S_{\lambda}} (Y^{(k)}_{r})^{2}
\zeta(\Xi_{r})^{4(p-1)+\eps -2} \| \nabla \zeta(\Xi_{r}) \sigma(\Xi_{r}) \|^{2} dr \right)^{1/2} .
\end{eqnarray*} 
But since $\nabla \zeta$ and $\sigma$ are bounded,
$$\bE \int_{0}^{t \wedge \S_{\lambda}} (Y^{(k)}_{r})^{2}
\zeta(\Xi_{r})^{\frac{4}{q}+\eps -2} \| \nabla \zeta(\Xi_{r})\sigma(\Xi_{r}) \|^{2} dr \leq C \bE \int_{0}^{\S} \zeta^{\eps-2}(\Xi_{r}) dr < +\infty.$$
Compared to \cite{popi:07}, the novelties are the generator $f$ and the terms $U^{(k)}$ and $M^{(k)}$. First using \eqref{eq:decomp_gene}:
\begin{align*} 
& - 2 \int_{0}^{t \wedge \S_{\lambda}} Y^{(k)}_{r} f(r,Y^{(k)}_{r},Z^{(k)}_{r},U^{(k)}_{r}) \zeta(\Xi_{r})^{4(p-1)+\eps} dr  \\
& = - 2 \int_{0}^{t \wedge \S_{\lambda}} Y^{(k)}_{r}f^0_r \zeta(\Xi_{r})^{4(p-1)+\eps} dr - 2 \int_{0}^{t \wedge \S_{\lambda}} Y^{(k)}_{r} \varrho(r,U^{(k)}_r) \zeta(\Xi_{r})^{4(p-1)+\eps} dr \\
&  - 2 \int_{0}^{t \wedge \S_{\lambda}} Y^{(k)}_{r} \varpi(r,Z^{(k)}_r, U^{(k)}_r) \zeta(\Xi_{r})^{4(p-1)+\eps} dr- 2 \int_{0}^{t \wedge \S_{\lambda}} Y^{(k)}_{r} \phi(r,Z^{(k)}_r, U^{(k)}_r) \zeta(\Xi_{r})^{4(p-1)+\eps} dr.
\end{align*} 
We know that $|Y^{(k)}_{r}f^0_r \zeta(\Xi_{r})^{4(p-1)+\eps}| \leq C$. 
From \ref{A4}
$$\varpi(r,Z^{(k)}_r, U^{(k)}_r)  =\varpi^{(k)}_r  Z^{(k)}_r$$
with $|\varpi^{(k)}_r | \leq L_z$. Again by the Cauchy-Schwarz inequality and the previous arguments:
\begin{align*}
&\left| \bE \int_{0}^{t \wedge \S_{\lambda}} Y^{(k)}_{r} \varpi(r,Z^{(k)}_r, U^{(k)}_r) \zeta(\Xi_{r})^{4(p-1)+\eps} dr \right| \\
& \quad \leq C \left( \bE \int_{0}^{t \wedge \S_{\lambda}} \|Z^{(k)}_{r} \|^{2} \zeta(\Xi_{r})^{4(p-1)+\eps} dr \right)^{1/2} .
\end{align*} 
From \ref{A3} and similar arguments, we also have:
\begin{align*}
&\left| \bE \int_{0}^{t \wedge \S_{\lambda}} Y^{(k)}_{r} \varrho(r,  U^{(k)}_r) \zeta(\Xi_{r})^{4(p-1)+\eps} dr \right| \\
& \quad \leq C \left( \bE \int_{0}^{t \wedge \S_{\lambda}}\int_{\cE} (U^{(k)}_{r}(e) )^{2} \mu(de) \zeta(\Xi_{r})^{4(p-1)+\eps} dr \right)^{1/2} .
\end{align*} 
Note that with \ref{B1}
$$ 2 \bE   \int_{0}^{t \wedge \S_{\lambda}} Y^{(k)}_{r} \phi(r,Y^{(k)}_{r} ,Z^{(k)}_{r} ,U^{(k)}_{r} ) dr \leq - 2 \bE   \int_{0}^{t \wedge \S_{\lambda}}\frac{1}{\eta_r} \left| Y^{(k)}_{r} \right|^q   dr \leq 0.$$
Up to some localization procedure we have
\begin{align*}
& \bE \int_{0}^{t \wedge \S_{\lambda}} \zeta (\Xi_{r})^{4(p-1)+\eps}  \int_{\cE} \left| U^{(k)}_r(e) \right|^2 \pi(de,dr) \\
& \quad =  \bE \int_{0}^{t \wedge \S_{\lambda}} \zeta (\Xi_{r})^{4(p-1)+\eps}  \int_{\cE} \left| U^{(k)}_r(e) \right|^2 \mu(de) dr .
\end{align*}
Coming back to \eqref{eq:Ito_form_control_Z_psi} and taking the expectation, we obtain:
\begin{align*} \nonumber
& \bE \left( Y^{(k)}_{t \wedge \S_{\lambda}} \right)^{2} \zeta(\Xi_{t \wedge \S_{\lambda}})^{4(p-1)+\eps} -\bE  \left( Y^{(k)}_{0} \right)^{2} \zeta(\Xi_{0})^{4(p-1)+\eps} \\
& \quad -  (4(p-1)+\eps)\bE \int_{0}^{t \wedge \S_{\lambda}} (Y^{(k)}_{r})^{2} \zeta(\Xi_{r})^{4(p-1)+\eps -1}\nabla \zeta(\Xi_{r})  b(\Xi_{r}) dr \\
& \quad - \frac{ (4(p-1)+\eps)}{2} \bE \int_{0}^{t \wedge \S_{\lambda}} (Y^{(k)}_{r})^{2} \left[ (4(p-1)+\eps -1) \zeta(\Xi_{r})^{4(p-1)+\eps -2} \|\sigma(\Xi_{r}) \nabla \zeta (\Xi_{r}) \|^{2} \right.\\ \nonumber
& \qquad \qquad \qquad \qquad \left.+ \zeta(\Xi_{r})^{4(p-1)+\eps -1} \tr \left( \sigma \sigma^{*} (\Xi_{r}) D^{2} \zeta(\Xi_{r}) \right) \right] dr \\ 
& \quad + 2\bE   \int_{0}^{t \wedge \S_{\lambda}} Y^{(k)}_{r} f^0_r \zeta(\Xi_{r})^{4(p-1)+\eps} dr \\ \nonumber
& \geq \bE \int_{0}^{t \wedge \S_{\lambda}} \| Z^{(k)}_{r} \|^{2} \zeta (\Xi_{r})^{4(p-1)+\eps} dr - C \left( \bE \int_{0}^{t \wedge \S_{\lambda}} \|Z^{(k)}_{r} \|^{2} \zeta(\Xi_{r})^{4(p-1)+\eps} dr \right)^{1/2}\\ \nonumber
& \quad  - C \left( \bE \int_{0}^{t \wedge \S_{\lambda}}\int_{\cE} (U^{(k)}_{r}(e) )^{2} \mu(de) \zeta(\Xi_{r})^{4(p-1)+\eps} dr \right)^{1/2} \\ \nonumber
& \quad + \bE \int_{0}^{t \wedge \S_{\lambda}} \zeta (\Xi_{r})^{4(p-1)+\eps}  \int_{\cE} \left|U^{(k)}_r(e) \right|^2 \mu(de) dr.
\end{align*} 
The left-hand side of the inequality is bounded, uniformly w.r.t. $k$, $t$ and $\lambda$. Hence for all $k$, $t$ and $\lambda$,
$$ \bE \int_{0}^{t \wedge \S_{\lambda}} \left( \| Z^{(k)}_{r} \|^{2} +  \int_{\cE} \left| U^{(k)}_r(e) \right|^2 \mu(de)  \right) \zeta (\Xi_{r})^{4(p-1)+\eps} dr \leq C.$$
By Fatou's lemma, 
$$ \bE \int_{0}^{\S} \left( \| \Zmin_{r} \|^{2} +  \int_{\cE} \left| \Umin_r(e) \right|^2 \mu(de)  \right) \zeta (\Xi_{r})^{4(p-1)+\eps} dr \leq C.$$
Since $\zeta \geq \dist$ on $\overline{D}$, we obtain the announced result. If {\bf (D)} holds, we adapt the above arguments using  \eqref{majortech} instead of \eqref{eq:upp_bound_rand_time}.
\end{proof}

\begin{theorem}
Assume that Conditions {\rm {\bf (A)}}, {\rm {\bf (B)}} and {\rm {\bf (D)}} hold. Then a.s.
$$\liminf_{t\to +\infty} \Ymin_{t\wedge \S} = \xi.$$
\end{theorem}
\begin{proof}
The proof is based on the arguments developed in \cite[Theorem 2]{popi:07} and \cite[Theorem 3.5]{popi:16}. Thus we skip the details and we only evoke the main ideas. 

Recall that $F_{\infty} = \left\{ g = + \infty \right\} \cap \partial D$ is a closed set, that $U$ is an bounded open set such that $\overline{U} \cap F_{\infty} = \emptyset$ and $U \cap \partial D \neq \emptyset$. Now we take a function $\varphi : \bR^{d} \to \bR_{+}$ of class $C^{2}$ and with a compact support included in $U$. For $\beta > 0$ we apply the It\^{o} formula to the process $e^{- \beta t} Y^{(k)}_{t} \varphi (\Xi_{t})$: 
\begin{align} \nonumber
& \bE \left[ e^{- \beta \S} (g\wedge k)(\Xi_{\S}) \varphi (\Xi_{\S}) \right] =  \bE \left[ e^{- \beta (t \wedge \S)}Y^{(k)}_{t \wedge \S} \varphi (\Xi_{t \wedge \S}) \right] \\ \nonumber
& -\beta  \bE \int_{t\wedge \S}^{\S} e^{-\beta s} Y^{(k)}_{s}  \varphi (\Xi_{s}) ds - \bE \int_{t\wedge \S}^{\S} e^{-\beta s}  \varphi (\Xi_{s})  f(s,Y^{(k)}_s,Z^{(k)}_s,U^{(k)}_s) ds \\ \label{eq:ito_cont_markov}
& +  \bE \int_{t\wedge \S}^{\S} e^{-\beta s} Y^{(k)}_s \cL \varphi (\Xi_{s}) ds  + \bE \int_{t\wedge \S}^{\S} e^{-\beta s} \nabla \varphi (\Xi_{s}) \sigma(\Xi_s)Z^{(k)}_s ds .
\end{align}
$\beta > 0$ is here only to avoid time integrability trouble. 
Again we decompose $f$ using \eqref{eq:decomp_gene}. 
\begin{align*}
& \bE \int_{t\wedge \S}^{\S} e^{-\beta s}  \varphi (\Xi_{s})  f(s,Y^{(k)}_s,Z^{(k)}_s,U^{(k)}_s) ds \\
& =  \bE \int_{t\wedge \S}^{\S} e^{-\beta s}  \varphi (\Xi_{s})  f^0_s ds +  \bE \int_{t\wedge \S}^{\S} e^{-\beta s}  \varphi (\Xi_{s}) \left[ f(s,Y^{(k)}_s,Z^{(k)}_s,U^{(k)}_s) - f(s,0,Z^{(k)}_s,U^{(k)}_s) \right] ds \\
& +  \bE \int_{t\wedge \S}^{\S} e^{-\beta s}  \varphi (\Xi_{s}) \varpi^{(k)}_s Z^{(k)}_s  ds + \bE \int_{t\wedge \S}^{\S} e^{-\beta s}  \varphi (\Xi_{s}) \varrho(s,U^{(k)}_s ) ds .
\end{align*}
Using the previous lemma and the Cauchy-Schwarz inequality, arguing as in \cite{popi:07}, we deduce the existence of some constant $C$, independent of $k$ and $t$, such that 
$$ \bE \int_{t\wedge \S}^{\S} e^{-\beta s} \left|  \varphi (\Xi_{s}) \left(  \varpi^{(k)}_s Z^{(k)}_s  + \varrho(s,U^{(k)}_s ) \right) + \nabla \varphi (\Xi_{s}) \sigma(\Xi_s)Z^{(k)}_s  \right| ds  \leq C.$$
From Lemma \ref{lem:major_Ymin_with_g},  
$$ \bE \int_{t\wedge \S}^{\S} e^{-\beta s} Y^{(k)}_{s}  \left| \varphi (\Xi_{s}) + \cL \varphi (\Xi_{s}) \right| ds  \leq C.$$ 
Hence all terms in \eqref{eq:ito_cont_markov}, except maybe
$$- \bE \int_{t\wedge \S}^{\S} e^{-\beta s}  \varphi (\Xi_{s}) \left[ f(s,Y^{(k)}_s,Z^{(k)}_s,U^{(k)}_s) - f(s,0,Z^{(k)}_s,U^{(k)}_s) \right] ds,$$
are uniformly bounded. Thus this remaining term is also bounded and, thanks to \ref{B1}, is greater than
$$ \bE \int_{t\wedge \S}^{\S} e^{-\beta s}  \varphi (\Xi_{s})(Y^{(k)}_s)^qds.$$
The dominated convergence theorem and again Lemma \ref{lem:major_Zmin_with_g} imply that, up to a suitable subsequence, we can pass to the limit on $k$ in \eqref{eq:ito_cont_markov} to obtain for any $t\geq 0$:
\begin{align*} \nonumber
& \bE \left[ e^{- \beta \S} g(\Xi_{\S}) \varphi (\Xi_{\S}) \right] =  \bE \left[ e^{- \beta (t \wedge \S)} \Ymin_{t \wedge \S} \varphi (\Xi_{t \wedge \S}) \right] \\ \nonumber
& -\beta  \bE \int_{t\wedge \S}^{\S} e^{-\beta s} \Ymin_{s}  \varphi (\Xi_{s}) ds - \bE \int_{t\wedge \S}^{\S} e^{-\beta s}  \varphi (\Xi_{s})  f(s,\Ymin_s,\Zmin_s,\Umin_s) ds \\ & +  \bE \int_{t\wedge \S}^{\S} e^{-\beta s} \Ymin_s \cL \varphi (\Xi_{s}) ds  + \bE \int_{t\wedge \S}^{\S} e^{-\beta s} \nabla \varphi (\Xi_{s}) \sigma(\Xi_s)\Zmin_s ds .
\end{align*}
Using Fatou's lemma and letting $t$ go to $+\infty$, we deduce that 
$$ \bE \left[ e^{- \beta \S} g(\Xi_{\S}) \varphi (\Xi_{\S}) \right] \geq \bE \left[ e^{- \beta \S}  \varphi (\Xi_{ \S})  \liminf_{t\to +\infty} \Ymin_{t \wedge \S}\right].$$
The conclusion follows since a.s. 
$$\liminf_{t\to +\infty} \Ymin_{t \wedge \S} \geq g(\Xi_{\S}) .$$
We emphasize again that the technical details are in \cite{popi:07,popi:16} and are skipped in this paper. Note that since $\Xi$ is continuous, several technical issues of \cite{popi:16} are avoided here. 
\end{proof}

\subsection{Related elliptic PDE}\label{ss:relatedPDE}

Since \cite{darl:pard:97,pard:99}, it is well known that BSDEs with random terminal time and elliptic PDE are strongly related. Inspiring by \cite{lega:97,marc:vero:98b,marc:vero:98a}, \cite{popi:07} extended such result to singular boundary value for the elliptic PDE, when the generator $f$ is of the form $-y|y|^{q-1}$, $q > 1$. Let us now assume that $\S$ is given by \eqref{eq:def_stop_time}, that $f$ is a deterministic function\footnote{If the terminal time and the terminal values are deterministic functions of $\Xi_\S$, then the solution of the BSDE \eqref{eq:bsde} verifies $U=M=0$. Hence we can assume w.l.o.g. that $f$ does not depend on $U$ here.}, and that the terminal time is given by \ref{D1}, namely $\xi = g(\Xi_{\S})$. 
We consider the system: for any $x \in D$
\begin{align} \label{eq:fbsde_sde}
\Xi^x_t& =x + \int_0^t b(\Xi^x_r)dr+\int_0^t \sigma(\Xi^x_u)dW_u, \\ \label{eq:fbsde_bsde}
\Yminx_t &  =g(\Xi^x_\S) +  \int_{t}^{\S} f(\Xi^x_r,\Yminx_r,\Zminx_r) dr-  \int_{t}^{\S} \Zminx_r dW_r .
\end{align}
Of course, Equation \eqref{eq:fbsde_bsde} of this forward-backward SDE has to be understood in the sense of Definition \ref{def:sol_sing_BSDE_rtt}. 

We consider the elliptic PDE 
\begin{equation} \label{eq:elliptic_pde}
\left\{ \begin{array}{rl}
- \cL v - f(x,v, \nabla v \sigma^*)= 0 & \mbox{on} \ D;\\
v = g & \mbox{on} \ \partial D,
\end{array} \right.
\end{equation}
where the operator $\cL$ is the infinitesimal generator of $\Xi$.

The following definition can be found in \cite{barl:93}, \cite{barl:94} (or \cite{pard:99}, \cite{cran:ishi:lion:92} for $v$ continuous). If $v$ is a function defined on $\overline{D}$, we denote by $v^{*}$ (respectively $v_{*}$) the upper- (respectively lower-) semicontinuous envelope of $v$: for all $x \in \overline{D}$
$$v^{*} (x) = \limsup_{x' \to x, \ x' \in \overline{D}} v(x') \quad \mbox{and} \quad v_{*} (x) = \liminf_{x' \to x, \ x' \in \overline{D}} v(x').$$
The next definition holds for bounded boundary condition $g$. 
\begin{definition}[Viscosity solution] \label{def:sol_visc}
\noindent \begin{itemize}
\item $v : \overline{D} \to \bR$ is called a \textbf{viscosity subsolution} of \eqref{eq:elliptic_pde} if $v^{*}< + \infty$ on
$\overline{D}$ and if for all $\phi \in
C^{2}(\bR^{d})$, whenever $x \in \overline{D}$ is a point of local maximum of $v^{*} - \phi$,
\begin{eqnarray*}
- \cL \phi (x) - f (x,v^{*}(x), \nabla \phi(x) \sigma^*(x))  \leq 0 & \mbox{if} & x \in D; \\
\min \left( - \cL \phi (x) - f (x,v^{*}(x), \nabla \phi(x) \sigma^*(x)), v^{*}(x) - g(x) \right) \leq 0 & \mbox{if} & x \in \partial D.
\end{eqnarray*}
\item $v : \overline{D} \to \bR$ is called a \textbf{viscosity supersolution} of \eqref{eq:elliptic_pde} if $v_{*} > - \infty$ on 
$\overline{D}$ and if for all $\phi \in
C^{2}(\bR^{d})$, whenever $x \in \overline{D}$ is a point of local minimum of $v_{*} - \phi$,
\begin{eqnarray*}
- \cL \phi (x) - f (x,v^{*}(x), \nabla \phi(x) \sigma^*(x)) \geq 0 & \mbox{if} & x \in D; \\
\max \left( - \cL \phi (x) - f (x,v^{*}(x), \nabla \phi(x) \sigma^*(x)), v(x) - g(x) \right) \geq 0 & \mbox{if} & x \in \partial D.
\end{eqnarray*}
\item $v : \overline{D} \to \bR$ is called a \textbf{viscosity solution} of \eqref{eq:elliptic_pde} if it is both a viscosity sub- and
supersolution.
\end{itemize}
\end{definition}
If the boundary condition is singular, we adapt the preceding definition. 
\begin{definition}[Unbounded viscosity solution] \label{def:unboundedviscsolellip}
We say that $v$ is a viscosity solution of the PDE \eqref{eq:elliptic_pde}
with unbounded terminal data $g$ if $v$ is a viscosity solution on $D$ in the sense of Definition \ref{def:sol_visc} and if 
$$g(x) \leq \lim_{{x' \to x}\atop{x' \in D, \ x \in \partial D}} v_{*}(x') \leq \lim_{{x' \to x}\atop{x' \in D, \ x \in
\partial D}} v^{*}(x') \leq g(x).$$ 
\end{definition}
Remark that this definition implies that $v^{*} < + \infty$ and $v_{*} > - \infty$ on $D$. Under Conditions {\bf (A)}, {\bf (B)} and {\bf (D)} and if 
\begin{itemize}
\item $g : \partial D  \to \overline{\bR_+}$ is continuous,
\item $f$ is continuous on $\overline{D} \times \bR \times \bR^d$, 
\end{itemize}
using \cite[Theorem 5.74]{pard:rasc:14}, if we define $u^{(k)}(x) = Y^{(k),x}_0$,
then $u^{(k)}$ is continuous on $\overline{D}$ and it is a viscosity solution of the elliptic PDE \eqref{eq:elliptic_pde} with boundary data $g \wedge k$. Evoke that the sequence $Y^{(k),x}$ is converging to $\Yminx$. If 
$$u(x) \triangleq \Yminx_{0},$$
then $u$ is the supremum of continuous functions $u^{(k)}$, $u$ is non-negative and lower-semicontinuous on $\overline{D}$ and satisfies:
\begin{equation*}
\forall x \in \overline{D}, \quad u(x) \leq \frac{C}{\dist^{2(p-1)}(x)}.
\end{equation*}
Following the arguments of \cite{popi:07} with some adapted modifications, we have:
\begin{proposition} \label{prop:visc_sol_sing_PDE}
If Conditions {\bf (A)}, {\bf (B)} and {\bf (D)} hold, and if $f$ and $g$ are continuous functions, then the function $u$ defined by $u(x) = \Yminx_0$ is a viscosity solution of the elliptic PDE in the sense of Definition \ref{def:unboundedviscsolellip}. 

Moreover suppose that the matrix $\sigma\sigma^*$ is uniformly elliptic: there exists a constant $\alpha > 0$ such that 
\begin{equation}\label{eq:unif_ellip}
\forall x \in \bR^{d}, \quad \sigma \sigma^{*}(x) \geq \alpha \mbox{Id}.
\end{equation}
If the map $(x,y,z)\mapsto (b(x),\sigma (x),f(x,y,z)) $ is of class $C^1$, then $u$ belongs to $C^0(\overline{D},[0,+\infty])) \cap C^2(D,[0,+\infty))$. 
\end{proposition}

\section{Terminal condition $\xi_1$} \label{sect:term_cond_xi_1}

In this section we study terminal conditions of the form
\[
\xi_1 = \infty \cdot {\bm 1}_{\{\tau \leq S\}}
\]
where $\tau$ is another stopping time. We know from \cite[Section 2]{ahmadi2019backward} that when $S =T$ is deterministic
and $\tau$ has a bounded density around the terminal time $T$,
the minimal supersolution of BSDE \eqref{eq:bsde} with terminal condition $\xi_1$ satisfies
\[
\lim_{t\rightarrow T} \Ymin_T = \xi_1.
\]
Our goal is to prove similar continuity results when $S$ is a stopping time. For this we will consider two approaches:
the first is an extension of the approach taken in \cite[Section 2]{ahmadi2019backward}, the present section focuses on this.
We consider a new approach in the next subsection.

\subsection{First approach} \label{ssect:xi1_v1}

The approach of \cite[Section 2]{ahmadi2019backward} can be summarized as follows: 
\begin{enumerate}
\item Assume that $\tau$ has a bounded density around the terminal time $T$.
\item Let $Y^\infty$ be the minimal supersolution of \eqref{eq:bsde} on the interval $[\![0,S]\!]$ with terminal condition $Y_S = \infty$;
define the auxiliary terminal condition 
\[
\xi_1^{(\tau)} =  {\bm 1}_{\{\tau \le S\}} Y^\infty_\tau.
\]
\item Use the bounded density assumption and apriori upperbounds on $Y^\infty_\tau$ to prove
\begin{equation}\label{e:Lrhobound}
{\mathbb E}\left[ \left( \xi_1^{(\tau)}\right)^\varrho \right] < \infty
\end{equation}
for some $\varrho > 1$, in particular, $\xi_1^{(\tau)}$ is not a singular terminal condition.
\item Let $\widehat Y^{u}$ be the solution of a linear BSDE with terminal condition $\xi_1^{(\tau)}$
whose driver term is chosen to guarantee  $\Ymin \le \widehat Y^{u}$ (the superscript $u$ stands for upper bound). 
\item Derive the continuity of $\Ymin$ from that of $\widehat Y^{u}.$
\end{enumerate}
This argument requires a modification when the terminal time $S$ is random because 1) apriori upperbounds on supersolutions
with explicit expressions
are not in general available and 2) even when such bounds were available, assumptions only on the distribution of $\tau$
(such as the bounded density assumption in the first item of the list above) would not be sufficient because
the expectation in \eqref{e:Lrhobound} depends on the joint distribution of $\tau$ and $Y^{\min}_{\tau\wedge S}.$
In light of these observations, in the next theorem we take \eqref{e:Lrhobound} as our starting point. 
Proposition \ref{p:examplexi1} gives an example of a case where \eqref{e:Lrhobound} is satisfied.
Let us emphasize that \eqref{e:Lrhobound} implies that $\bP(\tau=\S)= 0$. Indeed, if not, then 
$${\mathbb E}\left[ \left( \xi_1^{(\tau)}\right)^\varrho \right]  \geq {\mathbb E}\left[{\bm 1}_{\{\tau = S\}}   \left(  Y^\infty_\S \right)^\varrho \right] = +\infty.$$
\begin{theorem} \label{t:firsttheoremxi1}
Assume that the stopping time $\S$ is solvable, such that Conditions {\bf (A)} and {\bf (B)} hold. Let $\tau$ be a stopping time such that there exists $\varrho$ large enough (depending on $\delta$ and $\delta^*$ in {\rm \ref{B3}}) such that \eqref{e:Lrhobound} holds. Then $\Ymin$ is continuous at $\S$, that is a.s. 
$$\lim_{t\to +\infty} \Ymin_{t\wedge \S} = \xi_1.$$
\end{theorem}
\begin{proof}
We adopt the argument in \cite{ahmadi2019backward} given for deterministic terminal times (see the list above) to solvable terminal
times as follows. Since $\S$ is solvable, there exists a minimal supersolution $(Y^\infty,Z^\infty,U^\infty,M^\infty)$  to BSDE \eqref{eq:bsde} with terminal condition $+\infty$ at time $\S$. 

First, we consider the (linear in $y$) generator 
$$g(t,y,z,\psi)= \chi y + f( t ,0 ,z,\psi),$$
which satisfies all conditions {\bf (A)}, and the terminal value $\xi_1^{(\tau)}$ at the random time $\S$. Note that $\xi_1^{(\tau)}$ is $\cF_{\tau \wedge \S}$-measurable, thus $\cF_{\S}$-measurable. Let us check that \eqref{eq:int_cond_random_time} holds, namely for some $r > 1$ and $\rho > \nu(r)$
$$\bE \left[ e^{r\rho \S} |\xi_1^{(\tau)}|^r + \int_0^\S e^{r \rho t} |g(t,0,0,{\bm 0})|^r dt \right] < +\infty.$$
Note that $g(t,0,0,{\bm 0})=f^0_t$ and \ref{B2} holds. From the proof of \cite[Proposition 5]{krus:popi:15}, using \ref{B3}, there exists $r > 1$ and $\rho > \nu (r)$ such that $r\nu(r) < \delta$. Hence we can find $\gamma > 1$ such that $ \bE ( e^{r \rho \gamma \S}) < +\infty$. H\"older's inequality leads to:
$$\bE \left[ e^{r\rho \S} |\xi_1^{(\tau)}|^r \right] \leq \left( \bE e^{r \rho \gamma \S} \right)^{1/\gamma} \left( \bE |\xi_1^{(\tau)}|^{r \gamma*} \right)^{1/\gamma*}.$$
If $\varrho \geq r \gamma*$, then we deduce that $ \bE |\xi_1^{(\tau)}|^{r \gamma*} <+\infty$ and \eqref{eq:int_cond_random_time} is satisfied. 

Then we have to verify that \eqref{eq:int_cond_random_time_2} holds for $\xi_1^{(\tau)}$. This can be done by linearizing $g$ and using the same arguments as for \eqref{eq:int_cond_random_time}. 
Applying Theorem \ref{thm:rand_time_exist_sol_BSDE} leads to the existence and the uniqueness of the solution $(\widehat Y^{u},\widehat Z^{u},\widehat U^{u},\widehat M^{u})$. 

We next prove that $\widehat Y^{u}$ does serve as an upper bound on $Y^{(k)}$, the solution of the BSDE \eqref{eq:bsde} with terminal condition $\xi_1 \wedge k= k \mathbf 1_{\tau \le \S}$ at time $\S$: a.s. for any $t \geq 0$
$$Y^{(k)}_{t\wedge \tau\wedge \S} \le \widehat Y^{u}_{t\wedge \tau\wedge \S}.$$
Indeed by comparison principle, $Y^{(k)} \leq Y^\infty$. Hence a.s. $Y^{(k)}_{\tau \wedge \S} = Y^{(k)}_{\tau} \mathbf 1_{\tau \le \S} \leq Y^\infty_{\tau} \mathbf 1_{\tau \le \S} =\xi_1^{(\tau)}$.
Since $f(t,y,z,\psi)\leq g(t,y,z,\psi)$ by Condition \ref{A1}, we deduce the wanted result. 

We conclude using some linearization procedure (see \cite[Lemma 3]{ahmadi2019backward}) that a.s. on the $\cF_\S$-measurable set $\{\tau > \S\}$, that 
$$\lim_{t\to +\infty} \widehat Y^{u}_{t\wedge \S} =  0.$$
Thereby a.s. on the same set
$$0\leq \lim_{t\to +\infty} \Ymin_{t\wedge \S} \leq \lim_{t\to +\infty} \widehat Y^{u}_{t\wedge \S}= 0 =  \xi_1.$$
The continuity is proved. 
\end{proof}

%
%
%

Let us develop an example. Let us assume that $\S$ is the first exit time of $\Xi$ given by \eqref{eq:def_stop_time}, $\S = \S_D=\inf\{t\ge 0, \quad \Xi_t \notin D\}$, such that there exists a constant $C$ such that \eqref{eq:upp_bound_rand_time} holds:
\begin{equation*}
0 \leq  Y^\infty_{t\wedge \S} \leq \frac C{ \dist(\Xi_{t\wedge \S})^{2(p-1)}}.
\end{equation*}
We also suppose that $\sigma$ is uniformly elliptic (Equation \eqref{eq:unif_ellip}), such that by \cite{friedman}, for $\Xi_0 = x \in D$, $\Xi_t$ has a density $\phi(t,x,\cdot)$. 
Under this assumption,
to prove \eqref{e:Lrhobound} it suffices to prove
\begin{equation}\label{e:toprove}
{\mathbb E}\left[ {\bm 1}_{\{\tau \leq \S\}} 
\frac{1}{\text{dist}( \Xi_{\tau})^{\varrho 2(p-1)}} \right] < \infty,
\end{equation}
for some $\varrho > 1.$ 
Theorem \ref{t:firsttheoremxi1} above gives:
\[
\lim_{t\rightarrow \infty} \Ymin_{t\wedge \S} = \xi_1,
\]
{\em assuming} \eqref{e:toprove}.

%
%
%

The expectation in \eqref{e:toprove} depends on the joint distribution of $(\tau, S, \Xi_{\S})$. We are not aware of results
available in the current literature that would imply \eqref{e:toprove} under broad and general assumptions on these variables.
A basic case that can be treated with techniques that we know of is when $\tau$ is independent of $\Xi$ (and therefore of $S$). The next proposition proves
\eqref{e:toprove} under this setting.
\begin{proposition}\label{p:examplexi1}
Suppose that $\S$ is the first exit time of $\Xi$ given by \eqref{eq:def_stop_time}, that $\sigma$ is uniformly elliptic, and that $\tau$ is independent of $\Xi$. 
If $q > 1+2\varrho$, then
\begin{equation} \label{e:Lrhoboundprop}
{\mathbb E}\left[ {\bm 1}_{\{\tau \leq \S\}} 
\frac{1}{\dist( \Xi_{\tau})^{\varrho 2(p-1)}} \right] < \infty,
\end{equation}
\end{proposition}
\begin{proof}
The equality $1/p + 1/q =1$ and $ q> 1+2\varrho $ imply $2(p-1)\varrho < 1$. 
Let us denote the distribution of $\tau$ by $F_\tau$. The expectation \eqref{e:toprove} can then be written as
\begin{align}
{\mathbb E}\left[ {\bm 1}_{\{\tau \leq \S\}}  \notag
\frac{1}{\text{dist}( \Xi_{\tau})^{\varrho 2(p-1)}} \right] 
&= \int_0^\infty {\mathbb E}\left[ {\bm 1}_{\{t \leq \S\}} \frac{1}{\text{dist}( \Xi_{t})^{\varrho 2(p-1)}} \right] dF_\tau(t).
\end{align}
Since $\S$ is the exit time of $\Xi$ from a smooth domain with uniformly elliptic diffusion matrix, we have:
\begin{align}
{\mathbb E}\left[ {\bm 1}_{\{\tau \leq \S\}}  \notag
\frac{1}{\text{dist}( \Xi_{\tau})^{\varrho 2(p-1)}} \right] 
&= \int_0^\infty {\mathbb E}\left[ {\bm 1}_{\{t < \S\}} \frac{1}{\text{dist}( \Xi_{t})^{\varrho 2(p-1)}} \right] dF_\tau(t)
\intertext{that $\{\Xi_t \in D\} \supset \{t < S\}$ implies}
&\le \int_0^\infty {\mathbb E}\left[ {\bm 1}_{\{\Xi_{t} \in D\}} \frac{1}{\text{dist}( \Xi_{t})^{\varrho 2(p-1)}} \right] dF_\tau(t). \label{e:decompose}
\end{align}
We next bound
\[
{\mathbb E}\left[ {\bm 1}_{\{\Xi_{t} \in D\}} \frac{1}{\text{dist}( \Xi_{t})^{\varrho 2(p-1)}} \right].
\]
For $\Xi_0 = x \in D$,
let $\phi(t,x,\cdot)$ be the density of $\Xi_t$. The expectation above then can be written as
\begin{equation}\label{e:intrepexp1}
{\mathbb E}\left[ {\bm 1}_{\{\Xi_{t} \in D\}} \frac{1}{\text{dist}( \Xi_{t})^{\varrho 2(p-1)}} \right] =
\int_D \phi(t,x,y) \frac{1}{\text{dist}(y)^{\varrho 2(p-1)}} dy.
\end{equation}
Define $D_\epsilon = \{x \in D: \dist(x) \le \epsilon \}$ for $\epsilon >0$;
by \cite[Lemma 14.16]{gilb:trud:01} there exists $\epsilon_1' > 0$ such that $\dist$ is $C^2$ in $D_{\epsilon_1'}$.
Therefore one can choose $ \epsilon_1 \in (0,\epsilon_1']$ so that  $\text{dist}$ is smooth on $D_{\epsilon_1}$
and $x \notin D_{\epsilon_1}.$ The continuity of $\dist$ implies
that $D_\epsilon$ is closed; $D_\epsilon$ is therefore compact 
since $D_\epsilon \subset D$ and $D$ is bounded.
This, the continuity of $\dist$ and $x \notin D_{\epsilon_1}$ imply
\begin{equation}\label{d:C1}
C_1 \doteq \inf_{y\in D_{\epsilon_1}} |x-y|  > 0.
\end{equation}

Since $b$ and $\sigma$ are Lipschitz continuous and since $\sigma$ is uniformly elliptic, from \cite[page 16]{friedman} we have the following Aronson's estimate on
$\phi(t,x,y):$
\[
\phi(t,x,y) \le \frac{C_2}{t^{d /2}} e^{-\frac{\lambda_0|y-x|^2}{4t}}.
\]
This and \eqref{d:C1} imply
\[
\phi(t,x,y) \le \frac{C_2}{t^{d/2}} e^{-\frac{\lambda_0 C_1^2}{4t}},
\]
for $y\in D_{\epsilon_1}.$
The right side of this inequality is continuous and bounded for
$t \in [0,\infty]$. Therefore
\begin{equation}\label{d:C3}
C_3 \doteq \sup_{t \in [0,\infty], y \in D_{\epsilon_1}}
\phi(t,x,y) \le 
 \sup_{t \in [0,\infty], y \in D_{\epsilon_1}}
 \frac{C_2}{t^{d /2}} e^{-\frac{\lambda_0 C_1^2}{4t}} < \infty.
\end{equation}
We now decompose \eqref{e:intrepexp1} into two integrals
over $D_{\epsilon_1}$ and $D\setminus D_{\epsilon_1}$:
\begin{align}
&{\mathbb E}\left[ {\bm 1}_{\{\Xi_{t} \in D\}} \frac{1}{\text{dist}( \Xi_{t})^{\varrho 2(p-1)}} \right] = \notag
\int_D \phi(t,x,y) \frac{1}{\text{dist}(y)^{\varrho 2(p-1)}} dy\\
&~~=
\int_{D\setminus D_{\epsilon_1}} \phi(t,x,y) \frac{1}{\text{dist}(y)^{\varrho 2(p-1)}} dy.\notag
+
\int_{D_{\epsilon_1}} \phi(t,x,y) \frac{1}{\text{dist}(y)^{\rho 2(p-1)}} dy \\
&\le \frac{1}{\epsilon_1^{2\varrho(p-1)}} + 
\int_{D_{\epsilon_1}} \phi(t,x,y) \frac{1}{\text{dist}(y)^{\varrho 2(p-1)}} dy.\label{e:boundforfixedt1}
\end{align}
the last inequality coming from: $\dist(y) > \epsilon_1$ for $y \in D\setminus D_{\epsilon_1}$. 

It remains to bound the last integral. For this note that $\dist$ is $C^2$ over $D_{\epsilon_1}.$
Furthermore, $\partial D$ is the $0$-level curve of $\text{dist}$, in particular, for $y\in \partial D$, the gradient 
$\nabla \dist(y)$ is normal to $\partial D.$
$\partial D$ is a $C^1$ surface, with nonvanishing normal at everypoint. It follows from these and the definition
of $\text{dist}$ that $\nabla \text{dist}$ satisfies $|\nabla \text{dist}(y)| =1$ for $y\in \partial D.$
Now define 
\[
E_\epsilon = \{y \in D: \text{dist}(y) > \epsilon \} = D \setminus D_\eps.
\]
That $\text{dist}$ is $C^2( D_{\epsilon_1})$ implies that $\partial D_{\epsilon_1}$ is a $C^2$ bounded surface and that
the function
\[
A(\epsilon) = \text{Area}(\partial E_\epsilon)
\]
is $C^1$ over the interval $[0,\epsilon_1]$. In particular, it is continuous and satisfies
\begin{equation}\label{d:C4}
C_4 \doteq \sup_{\epsilon \in [0,\epsilon_1]} A(\epsilon) < \infty.
\end{equation}
This and
the definition of $\text{dist}$ imply $|\nabla\text{dist}(y)|=1$ for $y \in \partial D_\epsilon$ for $\epsilon \le \epsilon_1.$
We are now in a setting where we can apply the co-area formula \cite[Theorem 5, page 713]{evans2010partial}, which gives
\begin{align*}
\int_{D_{\epsilon_1}} \phi(t,x,y) \frac{1}{\text{dist}(y)^{\varrho 2(p-1)}} dy
&=\int_{0}^{\epsilon_1} \left(\int_{\partial E_\epsilon} \phi(t,x,y) dS\right)\frac{1}{\epsilon^{\varrho 2(p-1)}}d\epsilon
\intertext{ $\partial E_\epsilon \subset D_{\epsilon_1}$ and \eqref{d:C3} imply}
&\le \int_{0}^{\epsilon_1} \left(\int_{\partial E_\epsilon} C_3 dS\right)\frac{1}{\epsilon^{\varrho 2(p-1)}}d\epsilon
\intertext{This and \eqref{d:C4} give}
&\le C_3C_4\int_{0}^{\epsilon_1} \frac{1}{\epsilon^{\varrho 2(p-1)}}d\epsilon.
\end{align*}
Recall that $\varrho 2(p-1) < 1$. This and the last line imply
\begin{equation}\label{e:boundforfixedt2}
\int_{D_{\epsilon_1}} \phi(t,x,y) \frac{1}{\text{dist}(y)^{\varrho 2(p-1)}} dy < C_5,
\end{equation}
where
\[
C_5 \doteq C_3 C_4 \int_{0}^{\epsilon_1} \frac{1}{\epsilon^{\varrho 2(p-1)}}d\epsilon < \infty.
\]
The bound \eqref{e:boundforfixedt2} we have just derived and \eqref{e:boundforfixedt1} imply
\[
{\mathbb E}\left[ {\bm 1}_{\{\Xi_{t} \in D\}} \frac{1}{\text{dist}( \Xi_{t})^{\varrho 2(p-1)}} \right]
\le 
\frac{1}{\epsilon_1^{\varrho 2(p-1)}} +  C_5.
\]
This and \eqref{e:decompose} imply \eqref{e:Lrhoboundprop}.
\end{proof}

For example, if $f$ only depends on $y$ and is non increasing ($\chi=0$), then it is sufficient to have $q > 3$.

\subsection{A new argument for $\xi_1$}\label{ss:newxi1}

In the rest of the paper, to clearly state the ideas and for a less technical presentation
we will restrict our attention to the Brownian framework, i.e., we assume that $\bF = \bF^W$ is the filtration generated by
the $d$-dimensional Brownian motion $W$. Therefore \eqref{eq:bsde} reduces to \eqref{eq:bsdebrownian}, that is:
\begin{equation*}
dY_t  = - f(t,Y_t,Z_t) dt + Z_t dW_t.
\end{equation*}
The continuity arguments in Section \ref{ssect:xi1_v1} above and in \cite[Section 2]{ahmadi2019backward} 
use the solution of a linear auxiliary BSDE as an upper bound to the minimal supersolution. 
In this section we would like to explore a new upper bound that is based directly
on the original nonlinear BSDE. As will be seen, whenever applicable, 
this is more natural and leads to less strict conditions on the parameter $q$ of Condition \ref{B1}.
We assume $\tau$ and $S$ to be solvable in the sense of Definition \ref{d:solvable_rtt}.  
Let $Y^{S,\infty}$ and $Y^{\tau,\infty}$ denote the $\infty$-supersolutions\footnote{When we refer to $Y$ as the solution, 
we mean the first component $Y$ of a solution $(Y,Z)$.} corresponding to $\tau$ and $S$.
The main idea of the present section as compared
to that of Section \ref{ssect:xi1_v1} and \cite[Section 2]{ahmadi2019backward} is the following: we replace the upper bound process $\widehat Y^{u}$ of the proof of Theorem \ref{t:firsttheoremxi1} with $Y^{\tau,\infty}$.

\begin{theorem} \label{t:xi1new}
Suppose $\tau$ and $S$ are solvable in the sense of Definition \ref{d:solvable_rtt}.
Then a supersolution $\Ymin$ of \eqref{eq:bsdebrownian} with terminal condition $\Ymin_S =+\infty \cdot {\bm 1}_{\{\tau \le S\}}$
exists and 
\begin{equation}\label{e:continuityYminS}
\lim_{t\rightarrow \infty} \Ymin_{t\wedge S} =+\infty \cdot {\bm 1}_{\{\tau \le S\}} =\xi_1.
\end{equation}
\end{theorem}
\begin{proof}

By assumption there exists a supersolution $Y^{S,\infty}$ to the BSDE with terminal condition $Y_S = \infty$
and this supersolution is the limit of processes $Y^{(L)}$ which are solutions of the same BSDE with
terminal condition $Y_S = L.$ Let $\xi \ge 0$ be an arbitrary terminal condition and let $Y^{L,\xi}$ be the solution
of \eqref{eq:bsdebrownian} with terminal condition $Y_S = \xi \wedge L.$ Comparison with
$Y^{(L)}$ imply that $\lim_{L \nearrow \infty} Y^{L,\xi}$ defines the minimal supersolution $\Ymin$ to \eqref{eq:bsdebrownian} with terminal
condition $\xi.$  By assumption $\tau$ is solvable. Therefore, there exists a process $Y^{\tau,\infty}$ that is a supersolution
to BSDE \eqref{eq:bsdebrownian} with terminal condition $Y_\tau = \infty.$ Let $\tau_n$ be the sequence of increasing
stopping times in Definition \ref{def:sol_sing_BSDE_rtt} associated with this supersolution and let $Y^{\tau,\infty,L}$ be the
sequence of solutions of \eqref{eq:bsdebrownian} with terminal condition $Y_\tau = L$; by definition 
\[Y^{\tau,\infty} = \lim_{L \nearrow \infty} Y^{\tau,\infty,L}.\]
By Corollary \ref{coro:seq_rtt_bounded_Y}, $Y^{\tau,\infty}$ is bounded by $n$ in the interval $[\![0,\tau_n]\!].$

Similarly, let $Y^{S,\xi_1,L}$ be the sequence of solutions of BSDE \eqref{eq:bsdebrownian} with terminal condition
$Y_S = \xi_1 \wedge L = L \cdot {\bm 1}_{\{\tau \le S \}.}.$
We will now prove 
\begin{equation}\label{e:comparisonintaunS}
Y^{S,\xi_1,L}_t \le Y^{\tau,\infty}_t, \quad t \le \tau_n \wedge S.
\end{equation}
To prove this consider, for $L_1 > 0$ the solution $Y^{S,\xi_1,L,L_1}$ of BSDE \eqref{eq:bsdebrownian}
with terminal condition $Y_{\tau \wedge S } = \left( Y^{S,\xi_1,L}_{\tau} {\bm 1}_{\{\tau \leq S\}} \right) \wedge L_1 = (Y^{S,\xi_1,L}_{\tau} \wedge L_1)  {\bm 1}_{\{\tau \leq S\}}$, which is $\cF_{\tau \wedge \S}$-measurable. 
We will compare this process with $Y^{\tau,L_1}$, the solution of \eqref{eq:bsdebrownian} with terminal
condition $Y_{\tau} = L_1$, on the time interval $[\![0, \tau \wedge S]\!].$
By its definition,
the terminal value of $Y^{S,\xi_1,L,L_1}$ at time $\tau \wedge S$ equals,
\begin{align}\label{e:termcondYtauxi1LL1}
Y^{S,\xi_1,L,L_1}_{\tau \wedge S} &= (Y^{S,\xi_1,L}_{\tau} \wedge L_1)  {\bm 1}_{\{\tau \leq S\}} \notag 
\intertext{which is bounded by}
&\le L_1 {\bm 1}{\{\tau \leq S \}}.
\end{align}
Again by definition
$$Y^{\tau,L_1}_{\tau \wedge S}  = Y^{\tau,L_1}_{\tau} {\bm 1}_{\{\tau \leq S \}} + Y^{\tau,L_1}_S {\bm 1}_{\{S < \tau \}} = L_1 +  Y^{\tau,L_1}_S {\bm 1}_{\{S <  \tau \}}.$$
It follows from this $Y^{\tau,L_1} \ge 0$ and \eqref{e:termcondYtauxi1LL1} that
\begin{equation}\label{e:comparexi1LL1L1Sterm}
Y^{S,\xi_1,L,L_1}_{\tau \wedge S} \le Y^{\tau,L_1}_{\tau \wedge S}.
\end{equation}
The processes 
$Y^{S,\xi_1,L,L_1}$ and  $Y^{\tau,L_1}$ are solutions of BSDE \eqref{eq:bsdebrownian} on
the interval $[\![0,\tau \wedge S]\!]$ (in the sense of Theorem \ref{thm:rand_time_exist_sol_BSDE}). This, $\tau_n \wedge S \le \tau \wedge S$, 
the inequality \eqref{e:comparexi1LL1L1Sterm} and the comparison principle for BSDE
imply
\[
Y^{S,\xi_1,L,L_1}_{t} \le Y^{\tau,L_1}_{t}, \text{ for } t \in [\![0,\tau_n \wedge S]\!].
\]
Letting $L_1 \nearrow \infty$ gives \eqref{e:comparisonintaunS}.
Recall that $Y^{\tau,\infty}$ is bounded by $n$ in the interval $[\![0,\tau_n]\!].$ This and \eqref{e:comparisonintaunS}
implies the same bound for $Y^{S,\xi_1,L}$. Letting $L \nearrow \infty$ we discover that the process
$Y^{S,\xi_1}$ is a solution of \eqref{eq:bsdebrownian} in the interval $[\![0,\tau_n\wedge S]\!]$ with
terminal condition 
\begin{align*}
&\xi_1 {\bm 1}_{\{S < \tau_n\}} + Y^{S,\xi_1}_{\tau_n} {\bm 1}_{\{\tau_n \le S \}} = Y^{S,\xi_1}_{\tau_n} {\bm 1}_{\{\tau_n \le S \}}  \le n.
\end{align*}
 In particular, $Y^{S,\xi_1}$ is continuous on $[\![0,\tau_n\wedge S]\!]$ and satisfies
\[
\lim_{t\rightarrow \infty} Y^{S,\xi_1}_{t\wedge \tau_n\wedge \S} = Y^{S,\xi_1}_{\tau_n\wedge \S} -\Delta Y^{S,\xi_1}_{\tau_n\wedge \S}.
\]
Now over the event $\{\tau_n > \S \}$, $Y^{S,\xi_1}_{\tau_n\wedge \S}=0$, and since the filtration is continuous at time $\S$, there is no jump at time $\S$. Thus over the event $\{\tau_n > \S \}$
\[
\lim_{t\rightarrow \infty} Y^{S,\xi_1}_{t\wedge \tau_n\wedge \S} = 0.
\]
 Since $Y^{S,\xi_1}= \Ymin$, this and
\[
\bigcup_{n =1}^\infty \{\tau_n > \S  \} = \{\tau > S \}
\]
implies  \eqref{e:continuityYminS}.
\end{proof}

\subsection{An example in one space dimension}\label{ss:anexamplexi1}

In this subsection we go back to the setup studied in \cite[Section 2]{krus:popi:seze:18}: the driver is deterministic and only a function of $y$:
\[
f(y) = -y|y|^{q-1},
\]
the terminal time $S$ is deterministic $T$ and the terminal condition is
\begin{equation}\label{e:termcond1d}
Y_T = \infty \cdot {\bm 1}_{\{\tau \le T \}}
\end{equation}
where $\tau$ is the first exit time of $W$ from the interval $(0,L).$ Note that since $f$ is deterministic and since the terminal conditions only depend on $W$, the solution $(Y,Z,U,M)$ of BSDE \eqref{eq:bsde} is reduced to $(Y,Z,\mathbf{0},0)$ and the BSDE can be reduced to:
\begin{equation}\label{eq:bsdebrownian1d}
Y_s = Y_t + \int_s^t f(Y_r) dr  + \int_s^t Z_rdW_r. 
\end{equation}
Theorem 2.1 of \cite{krus:popi:seze:18} states
that for $q > 2$ the minimal supersolution of the BSDE \eqref{eq:bsdebrownian1d} with terminal condition \eqref{e:termcond1d} is continuous at time $T$.
Let $y_t$ denote the solution of $\frac{dy}{dt} = -f(y)$ on the interval $[0,T]$ with terminal
value $y_T = \infty$, i.e., 
\begin{equation}\label{e:trivialsol}
y_t \doteq ((q-1)(T-t))^{1-p},~~~ t < T,~~~ 1/p + 1/q = 1.
\end{equation}
The proof of \cite[Theorem 2.1]{krus:popi:seze:18} is based on the following
integrability result:
\begin{equation}\label{e:finitexp}
{\mathbb E}[y_\tau {\bm 1}_{\{\tau \leq T \}}] = {\mathbb E}[y_\tau {\bm 1}_{\{\tau < T \}}] < \infty.
\end{equation}
As in the proof of Theorem \ref{t:firsttheoremxi1}, \cite{krus:popi:seze:18} constructs a linear process that is continuous at time $T$
to find a continuous upperbound on the minimal supersolution (which implies the continuity of the minimal supersolution); 
the bound \eqref{e:finitexp} ensures that the upper bound linear process is well defined.
The bound \eqref{e:finitexp} requires $q > 2$ and that is the reason why this was assumed in \cite{krus:popi:seze:18}
in its treatment of the terminal condition \eqref{e:termcond1d}.
We will now derive the same continuity result under the assumption $q >1$ using Theorem \ref{t:xi1new} above.

To apply Theorem \ref{t:xi1new} to the present setup we need $T$ and $\tau$ to be solvable. This essentially means
that the BSDE has  weak supersolutions with terminal value $\infty$ at these terminal times.
The weak supersolution for terminal time $T$ is the deterministic process $t \mapsto y_t$.
That $\tau$ is solvable can be derived
from \eqref{eq:upp_bound_rand_time}. Instead of invoking this general result, in the following lemma we will make use of the simple
nature of $f$ and $W$ to explicitly construct the supersolution $Y^{\tau,\infty}$ with terminal condition $Y_\tau = \infty.$
Following \cite[page 307]{zaitsev2002handbook} we will use
\begin{equation}\label{d:bmx}
\bm{x}(v,v_l) \doteq v_l^{1-\frac{q+1}{2}}\left(\frac{q+1}{4}\right)^{1/2} \int_{1}^{v/v_l} \left(u^{q+1}-1\right)^{-1/2}du.
\end{equation}
to construct solutions to the ODE
\begin{equation}\label{e:ode}
\frac{1}{2} \dfrac{\mathrm{d}^2 V}{\mathrm{d}x^2} - V^q = 0.
\end{equation}
The function ${\bm x}$ is strictly increasing in $v$, furthermore, $q >1$ implies $\bm{x}(\infty,v_l) < \infty$.
Define
\[
{\bm L}(v_l) = {\bm x}(\infty,v_l).
\]
Let $\bm{x}^{-1}(\cdot,v_l)$ denote the inverse of $\bm{x}(\cdot,v_l).$
Now define
\[
{\bm v} (x,v_l) \doteq \bm{x}^{-1}(|x-L/2|,v_l).
\]
\begin{lemma}\label{l:bmx}
On the interval 
$[L/2 -{\bm L}(v_l), L/2 + {\bm L}(v_l)]$, ${\bm v} (\cdot,v_l)$ satisfies \eqref{e:ode} with boundary conditions $\infty$ on both sides.
\end{lemma}
\begin{proof} Direct calculation using the definition \eqref{d:bmx} of ${\bm x}.$
\end{proof}
To construct a supersolution of \eqref{eq:bsdebrownian1d}, we want to solve \eqref{e:ode} in the interval $[0,L]$
with $\infty$ terminal conditions. Note that ${\bm L}(0) = \infty$ and ${\bm L}(\infty) =0$ and ${\bm L}$ is a decreasing smooth function.
It follows that there is a unique $v^*$ such that ${\bm L}(v^*) = L/2.$ 
Then for $v_l = v^*$, ${\bm v}(x,v^*)$ solves \eqref{e:ode} in the interval $[0,L]$ with $\infty$ terminal conditions.
For our argument we also need solutions to \eqref{e:ode} in the time interval $[0,L]$ with boundary condition $n$ on both sides.
For this purpose, the next lemma constructs a sequence $ 0 < v_n \nearrow v^*$ such that ${\bm x}(n,v_n) = L/2.$
\begin{lemma}
There exists a sequence $ 0 < v_n \nearrow v^*$ such that ${\bm x}(n,v_n) = L/2.$
\end{lemma}
\begin{proof}
Recall that $v^*$ is the unique solution of ${\bm x}(\infty,v^*) = L/2$, i.e.,
\[
(v^*)^{1-\frac{q+1}{2}}\left(\frac{q+1}{4}\right)^{1/2} \int_{1}^{\infty} \left(u^{q+1}-1\right)^{-1/2}du = L/2.
\]
This implies in particular
\[
\bm{x}(1,v^*) =(v^*)^{1-\frac{q+1}{2}}\left(\frac{q+1}{4}\right)^{1/2} \int_{1}^{1/v^*} \left(u^{q+1}-1\right)^{-1/2}du < L/2.
\]
Furthermore, the function $v_l \mapsto {\bm x}(1,v_l)$ is continuous on $(0,v^*]$ and increases to $\infty$ as $v_l \searrow 0.$
This implies that there exists $v_1 < v^*$ satisfying $\bm{x}(1,v_1) = L/2.$ Now note $\bm{x}(2,v_1) > L/2$ and $\bm{x}(2,v^*) < L/2.$
Applying the same argument gives $v_2 \in (v_1,v^*)$ satisfying $\bm{x}(2,v_2) =L/2$. Repeating the same argument inductively
gives us an increasing sequence $v_n$ bounded by $v^*$ solving $\bm{x}(n,v_n) = L/2.$ The limit $v^{**}$ of this sequence
satisfies $\bm{x}(\infty,v^{**}) = L/2$. Recall that $v^*$ is the unique solution of this equation. This yields $v_n \nearrow v^*.$
\end{proof}
We can now state and prove the generalization of \cite[Theorem 4]{krus:popi:seze:18} to $q > 1$:
\begin{theorem}
For $q > 1$ the minimal supersolution of \eqref{eq:bsdebrownian1d} with terminal condition $Y_T = \infty \cdot {\bm 1}_{\{\tau \le T \}}$
is continuous at time $T$.
\end{theorem}
\begin{proof}
By the previous lemma there exists $v_n \nearrow v^*$ that solves ${\bm x}(n,v_n) = L/2$. It follows from this and
Lemma \ref{l:bmx} that  ${\bm v}(\cdot,v_n)$ solves \eqref{e:ode} on $[0,L]$ with terminal condition $n$ on both sides
and that ${\bm v}(\cdot, v_n) \rightarrow {\bm v}(\cdot, v^*)$. The comparison principle for the equation \eqref{e:ode}
implies that in fact  ${\bm v}(\cdot, v_n) \nearrow {\bm v}(\cdot, v^*)$. Now define the processes
\[
Y_t^{\tau,n} = {\bm v}(W_t,v_n) , Y_t^{\tau,\infty} = {\bm v}(W_t,v^*).
\]e:eq:bsdebrownian
It\^o's formula implies that $Y_t^{\tau,n}$ solves \eqref{eq:bsdebrownian1d} with terminal condition $Y_\tau = n$.
Define $\tau_n$ be the first time $W$ hits $[1/n,L-1/n]$. It\^o's formula implies $Y^{\tau,\infty}$ satisfies
\eqref{e:sdepart} (with $\beta_n = \tau_n$) and the definition of ${\bm v}(\cdot,v^*)$ and the continuity of the sample paths of $W$ imply
\eqref{eq:term_cond_super_sol} with $\xi = \infty$. Therefore, $Y^{\tau,\infty}$ is a weak supersolution of \eqref{eq:bsdebrownian1d}
with terminal condition $Y_\tau = \infty.$
Furthermore,  ${\bm v}(\cdot, v_n) \nearrow {\bm v}(\cdot, v^*)$ implies
$Y_t^{\tau,n} \nearrow Y_t^{\tau,\infty}$. These imply that $\tau$ satisfies all of the conditions of
being solvable. $T$ is also solvable because it is deterministic. Theorem \ref{t:xi1new} now implies the statement of the present theorem.
\end{proof}

\section{Terminal condition $\xi_2$} \label{sect:term_cond_xi_2}

We assume $S$ to be solvable. This means that there exists a minimal supersolution $Y^{S,\infty}\ge 0$
to \eqref{eq:bsdebrownian} with terminal condition $Y^{S,\infty}_S = \infty$ and a sequence of stopping times
$S_n \nearrow S$ such that $Y^{S,\infty}_t \le n$ for $t \le S_n.$ (Definitions \ref{def:sol_sing_BSDE_rtt} and \ref{d:solvable_rtt}, Lemma \ref{lem:solvability_min_sol} and Corollary \ref{coro:seq_rtt_bounded_Y}).

Our continuity result is as follows:
\begin{theorem}\label{t:xi2}
Suppose $S$ is solvable and $\tau$ is an arbitrary stopping time such that ${\mathbb P}(S= \tau) =0.$
Then BSDE \eqref{eq:bsdebrownian} has a supersolution in the time interval $[\![0,S]\!]$
with terminal condition $Y_S = \xi_2 = \infty \cdot {\bm 1}_{\{\tau > S\}}.$ Furthermore this supersolution
is continuous at $S$:
\begin{equation}\label{e:contxi2}
\lim_{t\rightarrow \infty} Y_{S \wedge t} = \xi_2.
\end{equation}
\end{theorem}
This generalizes \cite[Theorem 2]{ahmadi2019backward} which assumes deterministic terminal times, to random terminal times.
The main idea of the proof of \cite[Theorem 2]{ahmadi2019backward} generalized to the current setup is as follows: 
we construct a sequence of supersolutions to \eqref{eq:bsdebrownian} with terminal
conditions $ Y_{S} = \infty \cdot {\bm 1}_{\{\tau > S_n \}}$ where $S_n$ is the sequence of stopping times approximating $S$.
Note that these processes are all defined over the time interval $[\![0,S]\!]$, $S_n < S$ allows one to prove they are all continuous at time $S$. This, $\infty \cdot {\bm 1}_{\{\tau > S_n \}} \ge \infty \cdot {\bm 1}_{\{\tau > S \}}$
and comparison principle for BSDE allow one to argue that $Y^{S,\xi_2}$ is also continuous at $S$, which is the result we seek.

Let us define several processes that will be useful in the proof of Theorem \ref{t:xi2}, as solution of BSDE \eqref{eq:bsdebrownian} over the time interval $[\![0,S]\!]$, changing the terminal condition at time $\S$:
\begin{itemize}
\item $Y^{S,L}$ corresponds to the terminal condition $L$ ;
\item $Y^{S,0}$ to the  terminal condition $0$ ;
\item $Y^{S,\xi_2,L,n}$ to the terminal
condition $L \cdot {\bm 1}_{\{\tau > S_n \}}.$ 
\end{itemize}
Note that these terminal conditions are $\cF_\S$-measurable and bounded. Hence from Theorem \ref{thm:rand_time_exist_sol_BSDE} and the conditions {\bf (B)}, these solutions are well defined and unique (in the sense of Definition \ref{def:sol_BSDE_rtt}).

Let $Y^{S_n,\xi_2,L}$ be the solution of \eqref{eq:bsdebrownian} in the time interval $[\![0,S_n]\!]$ with terminal condition
\[
Y_{S_n} = Y^{S,L}_{S_n} \cdot {\bm 1}_{\{\tau > S_n\}} + Y^{S,0}_{S_n} \cdot {\bm 1}_{\{\tau \le S_n \}}.
\]
The existence and uniqueness of $Y^{S_n,\xi_2,L}$ comes from the estimates on $Y^{S,L}$ and $Y^{S,0}$ in Theorem \ref{thm:rand_time_exist_sol_BSDE}. 
We begin our argument with the following lemma. 
\begin{lemma}\label{l:structureofYSxi2Ln}
The process 
 $Y^{S,\xi_2,L,n}$ 
has the following structure:
\begin{equation}\label{e:structureofYSxi2Ln}
Y^{\S,\xi_2,L,n}_t = Y^{S_n,\xi_2,L}_t {\bm 1}_{t \le S_n} + 
Y^{\S,0}_t \cdot {\bm 1}_{t > S_n} \cdot {\bm 1}_{\tau \le S_n} + 
Y^{\S,L}_t \cdot {\bm 1}_{t > S_n} \cdot {\bm 1}_{\tau > S_n}.
\end{equation}
\end{lemma}
\begin{proof}
First, $S_n < S$ implies that the right side of \eqref{e:structureofYSxi2Ln} defines an adapted and continuous process, denoted by $\cY$, with bounded terminal condition $Y_\S = L \cdot {\bm 1}_{\{\tau > \S_n\}}$. Let us show that $\cY$ satisfies also \eqref{eq:bsdebrownian}. We define similarly:
\begin{equation*}
\cZ_t = Z^{S_n,\xi_2,L}_t {\bm 1}_{t \le S_n} + 
Z^{\S,0}_t \cdot {\bm 1}_{t > S_n} \cdot {\bm 1}_{\tau \le S_n} + 
Z^{\S,L}_t \cdot {\bm 1}_{t > S_n} \cdot {\bm 1}_{\tau > S_n}.
\end{equation*}
 For any $0\leq t \leq T$, let us distinguish several cases:
\begin{itemize}
\item If $0\leq t \leq T\leq \S_n < \S$, then since $Y^{S_n,\xi_2,L}$ solves \eqref{eq:bsdebrownian} on $[\![ 0, \S_n]\!]$:
\begin{align*}
\cY_{t\wedge \S} & = Y^{S_n,\xi_2,L}_t = Y^{S_n,\xi_2,L}_T + \int_t^T f(u,Y^{S_n,\xi_2,L}_u,Z^{S_n,\xi_2,L}_u) du - \int_t^T Z^{S_n,\xi_2,L}_u dW_u \\
&= \cY_{T\wedge \S} + \int_{t\wedge \S}^{T\wedge \S} f(u,\cY_u,Z^{S_n,\xi_2,L}_u) du - \int_{t\wedge \S}^{T\wedge \S} Z^{S_n,\xi_2,L}_u dW_u.
\end{align*}
\item If $\S_n < t \leq T$, then 
\begin{align*}
\cY_{t\wedge \S} & = Y^{\S,0}_{t\wedge \S}  \cdot {\bm 1}_{\tau \le \S_n} +  Y^{\S,L}_{t\wedge \S}  \cdot {\bm 1}_{\tau > \S_n} \\
&= \cY_{T\wedge \S} + \int_{t\wedge \S}^{T\wedge \S} \left[ f(u,Y^{\S,0}_u,Z^{\S,0}_u)   \cdot {\bm 1}_{\tau \le \S_n} + f(u,Y^{\S,L}_u,Z^{\S,L}_u) \cdot {\bm 1}_{\tau > \S_n} \right] du \\
& - \int_{t\wedge \S}^{T\wedge \S} \left[ Z^{\S,0}_u\cdot {\bm 1}_{\tau \le \S_n}   + Z^{\S,L}_u\cdot {\bm 1}_{\tau > \S_n} \right] dW_u \\
&= \cY_{T\wedge \S} + \int_{t\wedge \S}^{T\wedge \S} f(u,\cY_u ,Z^{\S,0}_u \cdot {\bm 1}_{\tau \le \S_n} + Z^{\S,L}_u\cdot {\bm 1}_{\tau > \S_n})du \\
& - \int_{t\wedge \S}^{T\wedge \S} \left[ Z^{\S,0}_u\cdot {\bm 1}_{\tau \le \S_n}   + Z^{\S,L}_u\cdot {\bm 1}_{\tau > \S_n} \right] dW_u 
\end{align*}
since both sets $\{\tau \le \S_n\}$ and $\{\tau < \S_n\}$ are $\cF_{\S_n}$-measurable. 
\item If $0\leq t \leq S_n <   T$, then 
\begin{align*}
\cY_{t\wedge \S} & = Y^{S_n,\xi_2,L}_t = \cY_{ \S_n} + \int_{t\wedge \S}^{\S_n} f(u,\cY_u,Z^{S_n,\xi_2,L}_u) du - \int_{t\wedge \S}^{ \S_n} Z^{S_n,\xi_2,L}_u dW_u \\
& = Y^{\S,0}_{\S_n}  \cdot {\bm 1}_{\tau \le \S_n} +  Y^{\S,L}_{\S_n}  \cdot {\bm 1}_{\tau > \S_n} \\
& + \int_{t\wedge \S}^{ \S_n} f(u,\cY_u,Z^{S_n,\xi_2,L}_u) du - \int_{t\wedge \S}^{ \S_n} Z^{S_n,\xi_2,L}_u dW_u \\
& = \cY_{T\wedge \S} + \int_{\S_n}^{T\wedge \S} f(u,\cY_u ,Z^{\S,0}_u \cdot {\bm 1}_{\tau \le \S_n} + Z^{\S,L}_u\cdot {\bm 1}_{\tau > \S_n})du \\
& - \int_{\S_n}^{T\wedge \S} \left[ Z^{\S,0}_u\cdot {\bm 1}_{\tau \le \S_n}   + Z^{\S,L}_u\cdot {\bm 1}_{\tau > \S_n} \right] dW_u \\
& + \int_{t\wedge \S}^{ \S_n} f(u,\cY_u,Z^{S_n,\xi_2,L}_u) du - \int_{t\wedge \S}^{ \S_n} Z^{S_n,\xi_2,L}_u dW_u . 
\end{align*}
\end{itemize}
Hence we have verified that $(\cY,\cZ)$ solves the BSDE \eqref{eq:bsdebrownian}. The statement of the lemma follows from the uniqueness of such a solution (Theorem \ref{thm:rand_time_exist_sol_BSDE}).
\end{proof}
We now give
\begin{proof}[Proof of Theorem \ref{t:xi2}]
Let $Y^{S,\xi_2 \wedge L}$ be the solution of \eqref{eq:bsdebrownian} with bounded terminal condition
$Y_S = \xi_2 \wedge L =  L \cdot {\bm 1}_{\{\tau > \S\}}.$ The inequality $\xi_2 \wedge L \le L$ implies
\[
Y^{S,\xi_2 \wedge L}_t \le Y^{S,L}_t, t \le \S.
\]
This and $Y^{S,L}_{t} \le n$ for $t \le \S_n$ imply that, 
if we define
\[
Y^{S,\xi_2}_t = \lim_{L \nearrow \infty} Y^{S,\xi_2\wedge L}_t,
\]
then 
$Y^{S,\xi_2}$
is a classical solution of \eqref{eq:bsde} on the time interval $[\![0,S_n]\!].$
That \eqref{e:contxi2} holds over the event $\{\xi_2 = \infty\} = \{ \tau > S\}$ follows from the fact
that $Y^{S,\xi_2}$ is constructed by approximation from below (see \cite{popi:06}). For completeness, 
we reproduced this argument:
note
\[
\liminf_{t\rightarrow \infty} Y^{S,\xi_2}_{t \wedge S} \ge \liminf_{t\rightarrow \infty} Y^{S,\xi_2 \wedge L}_{t\wedge S} = \xi_2 \wedge L
\]
for all $L$. Letting $L\nearrow \infty$ implies
\[
\liminf_{t\rightarrow \infty} Y^{S,\xi_2}_{t \wedge S} \ge \xi_2.
\]
In particular,
\[
\lim_{t\rightarrow \infty} Y^{S,\xi_2}_{t \wedge S} =\liminf_{t\rightarrow \infty} Y^{S,\xi_2}_{t \wedge S} = \xi_2 = \infty
\]
over the event $\{\tau > S\}.$ This proves \eqref{e:contxi2} over the event $\{\tau > S \}.$

It remains to prove \eqref{e:contxi2} over the event $\{\tau \le S \}.$
Recall the process $Y^{S,\xi_2,L,n}$ of \eqref{e:structureofYSxi2Ln} that is the solution of 
\eqref{eq:bsdebrownian} over the interval $[\![0,S]\!]$ with terminal condition $ Y_S  = L \cdot {\bm 1}_{\{\tau > S_n\}}.$
That $S_n \le S$ implies
\[
L \cdot {\bm 1}_{\{\tau > S_n\}} \ge L \cdot {\bm 1}_{\{\tau > S\}}.
\]
This and the comparison principle lead to
\[
Y^{S,\xi_2,L}_t \le Y^{S,\xi_2,L,n}_t,\quad \text{for } t \le \S.
\]
Lemma \ref{l:structureofYSxi2Ln} implies
\[
Y^{S,\xi_2,L,n}_t =Y^{S,0}_t, \text{for } t \in ]\!] S_n,S]\!] 
\]
over the event $\{\tau \le  S_n\}.$ Combining the last two displays we get
\[
Y^{S,\xi_2,L}_t \le Y^{S,0}_t, \text{for } t \in ]\!] S_n,S]\!] 
\]
over the event $\{\tau \le  S_n\}.$ The right side of the last inequality doesn't depend on $L$. Taking limits on the left gives
\[
Y^{S,\xi_2}_t \le Y^{S,0}_t, \text{for } t \in ]\!] S_n,S]\!].
\]
over the event $\{\tau \le  S_n\}.$  The right side of the above inequality is a classical solution of the BSDE \eqref{eq:bsdebrownian}
with $0$ terminal condition. Therefore, taking limits of both sides above give
\[
\limsup_{t\rightarrow \infty} Y^{S,\xi_2}_{t \wedge S} \le \lim_{t\rightarrow \infty} Y^{S,0}_{t\wedge S} = 0.
\]
By its construction, $Y^{S,\xi_2} \ge 0.$ This and the last display imply
\[
\lim_{t\rightarrow \infty} Y^{S,\xi_2}_{t \wedge S} = 0.
\]
over the event $\{\tau \le  S_n\}.$ 
Finally, $S_n \nearrow S$ and ${\mathbb P}(\tau = S) = 0$ imply 
 $\bigcup_{n=1}^\infty \{\tau \le  S_n\} = \{\tau \le S\}.$ This and the last display imply
\[
\lim_{t\rightarrow \infty} Y^{S,\xi_2}_{t \wedge S} = 0 = \xi_2
\]
over the event $\{\tau \le S \}.$ This completes the proof of the theorem.
\end{proof}

\section{Conclusion}\label{s:conc}
The present work develops solutions to the BSDE \eqref{eq:bsde} with
random terminal time $S$ for a range of singular terminal values.
We do this by proving that the minimal supersolution is continuous at $S$
and attains the terminal value. A key ingredient of our framework and
our arguments is
the concept of a solvable stopping time with respect to the given BSDE
and the filtration, introduced in the present work. Solvability
means that the
the BSDE has a supersolution with value $\infty$ at the given stopping time.
We note that
a stopping time that has a positive density around $0$ is not solvable.
We also note that deterministic times as well as exit times of continuous
diffusion processes from smooth domains are solvable. A natural direction
for future work is to further understand the concept of solvability and 
identify other classes of solvable/nonsolvable stopping times.

\bibliography{biblio}
\end{document}